\documentclass[final]{siamart171218}%
\newcommand{\TheTitle}{Anderson-accelerated convergence of Picard iterations for 
incompressible Navier-Stokes equations}

\newcommand{\TheAuthors}{S. Pollock and L. G. Rebholz and M. Xiao}
\headers{\TheTitle}{\TheAuthors}

\ifpdf
\hypersetup{
  pdftitle={\TheTitle},
  pdfauthor={\TheAuthors}
}
\fi
\title{{\TheTitle}\thanks{Submitted to the editors \today.
\funding{SP was supported in part by NSF DMS 1719849.
         LR was supported in part by NSF DMS 1522191.
         MX was supported in part by NSF DMS 1522191.}}}

\author{Sara Pollock
\thanks{Department of Mathematics,
         University of Florida,
         Gainesville, FL, 32611 (\email{s.pollock@ufl.edu}).}
\and
Leo G. Rebholz
\thanks{
Department of Mathematical Sciences,
Clemson University, 
Clemson, SC 29634 (\email{rebholz@clemson.edu}).}
\and
Mengying Xiao
\thanks{
Department of Mathematics,
College of William \& Mary, 
Williamsburg, VA 23185 (\email{mxiao01@wm.edu}).}
}
\usepackage{float} 
\usepackage{fullpage}
\usepackage{amsfonts}
\usepackage{amsmath}
\usepackage{amssymb}
\usepackage{amsbsy}
\usepackage{mathtools}
\usepackage{epsfig}
\usepackage{graphicx}
\usepackage{colordvi}
\usepackage{graphics}
\usepackage{color}
\usepackage{bm}
\newtheorem{remark}[theorem]{Remark}
\newtheorem{assumption}[theorem]{Assumption}
\newtheorem{alg}[theorem]{Algorithm}
\newcommand{\norm}[1]{\left\Vert#1\right\Vert}
\newcommand{\ns}[1]{\norm{#1}_*}

\newcommand{\Ome}{\Omega}
\newcommand{\nab}{\nabla}
\newcommand{\bv}{ v}
\newcommand{\bw}{ w}
\newcommand{\bX}{ X}

\newcommand{\bH}{ H}

\newcommand{\p}{\partial}
\newcommand{\bu}{ u}

\newcommand{\te}{\tilde e}
\newcommand{\tu}{\tilde u}

\newcommand{\ua}{u^\alpha}
\newcommand{\nr}[1]{\ensuremath{\left\|{#1} \right\|}}
\newcommand{\grad}{\nabla}

\newcommand{\goto}{\rightarrow}

\newcommand{\bigo}{{\mathcal O}}

\newcommand{\f}{\frac}
\newcommand{\argmin}[1]{\underset{#1}{\operatorname{arg}\,\operatorname{min}}\;}

\begin{document}

\maketitle

\begin{abstract}
We propose, analyze and test Anderson-accelerated Picard iterations for solving the 
incompressible Navier-Stokes equations (NSE).  Anderson acceleration has recently gained 
interest as a strategy to accelerate linear and nonlinear iterations, based on including 
an optimization step in each iteration.  We extend the Anderson-acceleration theory to 
the steady NSE setting and prove that the acceleration improves the convergence rate of 
the Picard iteration based on the success of the underlying optimization problem. 
The convergence is demonstrated in several numerical tests,
with particularly marked improvement in the higher Reynolds number regime.
Our tests show it can be an enabling technology in the sense that it can provide 
convergence when both usual Picard and Newton iterations fail.  
\end{abstract}



\section{Introduction}
We consider numerical solvers for the steady incompressible Navier-Stokes equations (NSE), 
which are given in a domain $\Omega\subset \mathbb{R}^d$  (d=2,3) by
\begin{eqnarray}
u \cdot\nabla u + \nabla p - \nu\Delta u  & = & f, \label{ns1} \\
\nabla \cdot u & = & 0, \label{ns2} \\
u|_{\partial\Omega} & = & g, \label{ns3} 
\end{eqnarray}
where $\nu$ is the kinematic viscosity, $f$ is a forcing, and $u$ and $p$ represent velocity and pressure.  For simplicity of our presentation and analysis, we consider homogeneous Dirichlet boundary
conditions, i.e. $g= 0$, but our theory can be extended to other common boundary conditions.

We study herein an acceleration technique applied to the Picard method for solving
the steady NSE.  The Picard method is commonly used for solving the steady NSE due to its stability
and global convergence properties, and takes the form (suppressing a spatial discretization)
\begin{eqnarray}
u_k \cdot\nabla u_{k+1} + \nabla p_{k+1} - \nu\Delta u_{k+1}  & = & f, \label{p1} \\
\nabla \cdot u_{k+1} & = & 0, \label{p2} \\
u_{k+1}|_{\partial\Omega} & = & 0, \label{p3} 
\end{eqnarray}
This iteration can be written as a fixed point iteration, $u_{k+1}=G(u_k)$, with $G$ denoting a solution operator for the Picard linearization \eqref{p1}-\eqref{p3}.

In practice, unfortunately, the Picard iteration often converges slowly, sometimes so 
slowly that for all practical purposes it fails.  
To improve this slow convergence, we employ an acceleration strategy introduced by D.G. 
Anderson in 1965 \cite{anderson65}.  In recent years, this strategy now
commonly referred to as Anderson acceleration has been analyzed in the context of 
multisecant methods for fixed-point iterations in \cite{FaSa09} motivated by a 
problem in electronic structure computations; and, in the context of generalized minimal residual (GMRES) methods in \cite{WaNi11}, where the efficacy of the method is 
demonstrated on a range of nonlinear problems.  
We further refer readers to \cite{K18,LWWY12,WaNi11} and the references
therein for detailed discussions on both practical implementation and 
a history of the method and its applications. Despite its long history of use, 
the first convergence analysis for Anderson acceleration 
(in both the linear and nonlinear settings) appears in 2015 in \cite{ToKe15}, 
under the usual local assumptions for convergence of Newton iterations.  
However, this theory (which we summarize in Section 2) does not prove 
that Anderson acceleration actually improves the convergence of a fixed point iteration.

The main contributions of this work involve Anderson acceleration applied to the Picard 
iteration for the steady NSE.  In this setting, we are able to prove that Anderson 
acceleration gives guaranteed improvement over the usual Picard iteration 
in a neighborhood of the fixed-point.  To our knowledge, this is the first
proof of improved convergence for Anderson acceleration applied to a nonlinear fixed 
point iteration, and thus may give insight into how a theory for general nonlinear 
fixed point operators might be developed.  Additionally, we show with several numerical 
experiments that Anderson acceleration can provide dramatic improvement in the Picard 
iteration, and can even be an enabling technology in the sense that it provides 
convergence in cases where both Picard and Newton fail.
In addition to this result, we also investigate the global convergence
behavior of Anderson acceleration for contractive operators. We find a  relation between 
the gain from the optimization, bounds on the optimization coefficients and the 
convergence rate of the underlying fixed-point iteration that assures the accelerated
sequence converges at an improved rate, independent of the initial error. 

This paper is arranged as follows.  
In \S \ref{sec:anderson} 
we provide some background on Anderson acceleration and its convergence
properties, and show global $r$-linear convergence at an improved rate based on success 
of the optimization problem for small enough coefficients.
In \S \ref{sec:PicardNSE} 
we give preliminaries for the steady NSE and associated finite element spatial 
discretization, and provide details of properties of the solution operator of the
fixed-point iteration associated with the discrete Picard linearization of the steady NSE.
In \S \ref{sec:AAP-NSE} 
we then analyze the Anderson accelerated Picard iteration for the steady NSE.  
We extend the general convergence results of \cite{K18,ToKe15} to this problem, and for 
the $m=1$ and $m=2$ cases, prove that Anderson acceleration improves the contraction 
ratio of the Picard iteration.  
In \S \ref{sec:numerics} 
 we report on results of several numerical tests for Anderson accelerated
Picard iterations for the steady NSE, and show that it can have a dramatic positive 
impact. 

\section{Anderson acceleration}\label{sec:anderson}
We discuss now the general Anderson acceleration algorithm and its 
convergence properties for contractive nonlinear operators.  In later
sections, we will consider the specific case of Picard iterations for the 
steady incompressible NSE.  
We start by stating the algorithm and reviewing the relevant known theory.
Theorem \ref{thm:reduc} is a new contribution to the theory for general nonlinear 
contractive operators. 
It shows that Anderson acceleration increases the convergence rate of the fixed-point
iteration when the optimization coefficients satisfy certain bounds.
We begin with the basic assumption of a contractive (nonlinear) operator.

\begin{assumption}\label{assume:G} 
Let $G: X \rightarrow X$ be a contractive operator with contraction ratio $r<1$, i.e.
\[
\| G(u)-G(w)\|_* \le r \| u-w\|_*, \quad \forall u,w \in X, 
\]
for a given space $X$ with norm $\| \cdot \|_*$ .
\end{assumption}
By standard fixed-point theory, under Assumption \ref{assume:G} there exists a unique 
$u^*\in X$ such that $G(u^*) = u^*$.  Although in \S 3 and beyond we will make specific
choices for $G$ and $X$, we discuss the acceleration algorithm in this form 
to emphasize its the more general applicability.

\begin{alg}[Anderson iteration] \label{alg:anderson}
The Anderson-acceleration with depth $m$ reads: \\ 
Step 0: Choose $u_0\in X.$\\
Step 1: Find $\tilde u_1\in X $ such that $\tilde u_1 = G(u_0)$.  
Set $u_1 = \tilde u_1$. \\
Step $k$: For $k+1=1,2,3,\ldots$ Set $m_k = \min\{ k, m\}.$\\
\indent [a.] Find $\tilde u_{k+1} = G(u_k)$. \\
\indent [b.] Solve the minimization problem for $\{ \alpha_{j}^{k+1}\}_{k-m_k}^k$
\[
\min_{\sum\limits_{j=k-m_k}^{k} \alpha_j^{k+1}  = 1} 
\left\|   \sum\limits_{j=k-m_k}^{k} \alpha_j^{k+1}( \tilde u_{j+1} - u_j) \right\|_* .
\] 
\indent [c.] Set  $u_{k+1} =  \sum\limits_{j= k-m_k}^k \alpha_j^{k+1} \tilde u_{j+1} $.
\end{alg}
\begin{remark}
For the more general Anderson mixing algorithm, set
$u_{k+1}$ in Algorithm \ref{alg:anderson} by 
$$u_{k+1} = \beta_{k+1} \sum\limits_{j= k-m_k}^k 
\alpha_j^{k+1} \tilde u_{j+1} + (1-\beta_{k+1}) 
\sum\limits_{j = k-m_k}^k \alpha_j^{k+1} u_j,$$  
for damping parameter $0 < \beta_{k} \le 1$.
Here we consider the undamped case $\beta_k=1$ for all $k$.
\end{remark}

The convergence of Anderson acceleration is studied in 
\cite{K18,ToKe15}, and for general nonlinear $G$ it is 
known that in a small enough neighborhood of the solution, the acceleration will 
not make the convergence significantly worse. To our knowledge however  
there is no mathematical proof that Anderson acceleration increases the convergence 
compared to the associated fixed point iteration.  The following result is proven in Theorem 2.3 in \cite{ToKe15}, and is the best known result for (locally) contractive 
operators.

\begin{theorem}[Convergence of Anderson acceleration] \label{thm:conv1} 
Assume operator $G$ has fixed-point $u^\ast$, and satisfies the following two conditions
under some norm $\ns{\,\cdot\,}$.
\begin{enumerate}
\item $G$ is Lipschitz continuously differentiable in a ball 
$\mathcal{B}(\rho) = \{ u\in X_h : \| u - u^\ast\|_* < \rho \}$ for some $\rho>0$,
\item There is a $c\in (0,1)$ such that for all $u,v\in \mathcal{B}(\rho)$, 
$\| G(u) - G(v) \|_\ast \le c \| u - v \|_\ast$.
\end{enumerate}
Then if $\sum_{j=1}^{m_k} |\alpha_j^k|$ is uniformly bounded for all $k>0$,  
Algorithm \ref{alg:anderson} converges to $u^*$ with contraction ratio $\hat c $ where 
$c< \hat c <1$, provided $\norm{u_0-u}_\ast$ is small enough.
\end{theorem}

We improve on this result for steady NSE in \S \ref{sec:AAP-NSE} 
where we show for the contractive operator $G$ 
associated with the Picard iteration that the convergence of the 
residual to zero is guaranteed to be accelerated close enough to the solution.  
While this result depends on the particular structure of the steady NSE and cannot be 
immediately applied to general contractive operators, the tools we employ may
give insight into how a more general result of improved convergence rate
can be constructed. 

Under some stronger assumptions on the coefficients $\alpha$ of the minimization
step,  we next establish a globally accelerated rate of convergence of the error
for general contractive operators. 
The idea of this analysis is to characterize the improvement
in the convergence rate by the balance between the success of the optimization problem 
solved at each step and the magnitude of the coefficients corresponding to earlier
solutions. The common link between the analysis here and in \S \ref{sec:AAP-NSE} is in characterizing the improvement
in convergence rate by the gain from the optimization problem.
We now fix some notation used in the remainder of the article.
\begin{align}\label{eqn:notation}
e_k \coloneqq u_k - u_{k-1},
\quad \te_k \coloneqq \tu_k - \tu_{k-1}, 
\quad w_k \coloneqq G(u_k) - u_k.
\end{align}

To aid in the analysis here and in \S \ref{sec:AAP-NSE}
we introduce an intermediate quantity 
\begin{align}\label{eqn:ua_k}
\ua_k = \sum_{j = k-m_k}^k \alpha^{k+1}_j u_j.
\end{align}
In particular, $\ua_k$ satisfies 
$\| u_{k+1} - \ua_k \|_* = \theta_k \nr{\tu_{k+1}- u_k}_*$,  
where $0 < \theta_k \le 1$ denotes the gain of the optimization of Step $k [b.]$ by
\begin{align}\label{eqn:n3-002}
\min\limits_{\sum\limits_{j=k-m_k}^k\alpha^{k+1}_j = 1}  
\nr{\sum_{j = k-m_k}^k \alpha_j^{k+1}(\tilde u_{j+1} - u_j)}_*
= \theta_k \nr{\tu_{k+1}- u_k}_*.
\end{align}
As $\theta_k = 1$ corresponds to the original fixed-point iteration, 
it is expected that $\theta_k<1$ for all $k$.

\begin{theorem}\label{thm:reduc}
Let the sequences $\{u_k\}$ and $\{\tu_k\}$ be given by Algorithm \ref{alg:anderson}.  
Let $G$ satisfy Assumption \ref{assume:G}.
Suppose the first $m_{k}$ coefficients of each $\alpha_j^{k+1}$ satisfy
$\left| \sum\limits_{j = k-m_k}^{l} \alpha_j^{k+1}\right| \le \eta$, 
$l = k-m_k, \ldots, k-1$, for some $0 < \eta < 1$. 
Define $e_k$ as in \eqref{eqn:notation}.
Then $\ns{e_2} \le (\kappa \theta_1 + \eta) \ns{e_1}$ and 
it holds for $2 \le k \le m$ that
\begin{align}\label{eqn:n3-007}
\nr{e_{k+1}}_* 
\le (r \theta_k + \eta) \nr{e_k}_*
  + \eta(r \theta_k +1) \sum_{j = 1}^{k-1} \ns{e_{j}}. 
\end{align}
For $k > m$, ($m_{k-1} = m_k =m$) it holds that
\begin{align}\label{eqn:n3-008}
\nr{e_{k+1}}_* 
 \le (r \theta_k + \eta)\nr{e_k}_* 
  + \eta(r \theta_k +1) \sum_{j = k-m+1}^{k-1} \ns{e_{j}}
  + r \theta_k \eta \ns{e_{k-m}},
\end{align}
where the sums are understood to be zero if the final index is less that the starting
index. 
\end{theorem}

The above theorem shows that if $\eta$ is small (requiring $\{\alpha_k^{k+1}\}$ 
close to 1), then Algorithm \ref{alg:anderson} can speed up convergence.  
The precise relationship between $r, \theta$ and $\eta$ to assure
$r$-linear convergence at a rate greater than $r$ is given in the corollary that follows.
This estimate also suggests one of they ways the accelerated algorithm can stall by 
failing to increase or even maintain the standard fixed-point 
convergence rate if coefficients $\alpha^{k+1}_j, ~j \le k-1,$ 
corresponding to iterates earlier in the history are too large. 

\begin{proof}
The proof makes use of the decomposition
\begin{align}\label{eqn:n3-001}
\nr{u_{k+1} - u_k}_* \le \nr{u_{k+1} - \ua_k}_* + \nr{ \ua_k - u_k}_*.
\end{align}
Expanding $u_k$ as a linear combination of $G(u_j)$, $j = k-1-m_{k-1}, \ldots, k-1$,
using the property that the coefficients of $\alpha_j^{k}$ sum to unity and telescoping
the resulting difference, we have
\begin{align}\label{eqn:n3-003}
\nr{G(u_k) - u_k}_* & = 
\nr{\sum_{j = k-1 - m_{k-1}}^{k-1} \alpha^k_j(G(u_k) - G(u_j)) }_*
\nonumber \\
& = \nr{ \sum_{j = k-m_{k-1}}^{k}  \left( \sum_{n=k-m_{k-1}-1}^{j-1} \alpha_n^k\right)
         \left(G(u_{j}) - G(u_{j-1}) \right)} 
\nonumber \\
& \le  \nr{G(u_k) - G(u_{k-1})} _*
 + \eta \sum_{j = k-m_{k-1}}^{k-1} \nr{G(u_{j}) - G(u_{j-1}) }_*
\nonumber\\
& \le  r \left( \nr{e_k}_* 
 + \eta \sum_{j = k-m_{k-1}}^{k-1} \nr{e_{j} }_* \right),
\end{align}
where the last inequality follows from the Lipschitz property of $G$.
By the same reasoning as above
\begin{align}\label{eqn:n3-005} 
\nr{\ua_k - u_k}_* 
 =  \nr{\sum_{j = k-m_k+1}^{k} \left( \sum_{n = k-m_k}^{j-1}\alpha_n^{k+1} \right) e_j}_*
 \le  \eta \sum_{j = k-m_k+1}^{k} \nr {e_j}_*.
\end{align}
Putting \eqref{eqn:n3-002}, \eqref{eqn:n3-003} and \eqref{eqn:n3-005} together into
\eqref{eqn:n3-001}
establishes the result.
\end{proof}

Theorem \ref{thm:reduc} gives an essential worst-case scenario where no cancellation
between the iterates is accounted for. Nonetheless, for a given bound $\eta$ 
we can determine sufficient optimization gain $\theta$ to ensure $r$-linear 
convergence $\nr{e_{k+1}}_* \le r^k\nr{e_1}_*$ where $r$ is the convergence rate of 
the underlying fixed-point iteration. A similar formula can be derived for 
$r$-linear convergence at a given rate $q$.
\begin{corollary}\label{cor:qfac}
Let the sequence $\{u_k\}$ be given by Algorithm \ref{alg:anderson} and 
suppose the hypotheses of Theorem \ref{thm:reduc} hold true.
Then $r$-linear convergence with factor $r$ holds for $k \ge 1$
\begin{align}\label{eqn:n3c1-00}
\nr{u_{k+1} - u_k}_* &\le r^k \nr{u_1 - u_0}_*,
\end{align}
if it holds that $\theta_1< 1-\eta/r$ and,
\begin{align}\label{eqn:n3-009}
\theta_k \le 
\left\{ \begin{array}{cc}
    \left(\f{r^{k} - \eta (1-r^{k})/(1-r)  }
                {r^{k} + \eta (r-r^{k})/(1-r)  }  \right), & k \le m 
\\
    \left(\f{r^{m} - \eta (1-r^{m})/(1-r)  }
            {r^{m} + \eta (1-r^{m})/(1-r)  }  \right), & k > m,
\end{array} \right.
\end{align}
and $\eta < r^m(1-r)/(1-r^m)$.
\end{corollary}
For instance, with $r = 0.9$ and $\eta = 0.1$, we have for the $m=1$ case
$\nr{e_{k+1}}_\ast \le r^k \nr{e_1}_\ast$ for $\theta_1 = 8/9$ and 
$\theta_k \le (r - \eta)/(r + \eta) = 0.8, k > 1$.
For $m=2$ we require $\theta_k \le 0.62$ for $k > 2$.
The proof follows directly from the result of Theorem \ref{thm:reduc} by induction on 
$k$, first for $k \le m$, then for $k > m$, and is left to the interested reader.

The relevance of this result is that it quantifies a relation between the parameters of
the optimization and the contractive operator for which global convergence at a given 
rate will be observed. In contrast, the results in section \S \ref{sec:AAP-NSE} 
and those in \cite{K18,ToKe15} prove an accelerated rate of convergence 
only once the residual is small enough.
Corollary \ref{cor:qfac} encompasses the preasymptotic regime,
describing the global convergence seen in \S \ref{sec:numerics}; and, is consistent with
results of  \cite{LWWY12} for finite difference approximations to
Richard's equation in which a lack of significant dependence on choice of initial iterate
is demonstrated numerically.

\section{The Picard iteration for steady NSE}\label{sec:PicardNSE}
We next consider the steady incompressible NSE.  
First, we give the mathematical framework and define some notation including the Picard 
iteration and associated Picard solution operator.
Then we prove two important properties for the solution operator in order to relate
it to the developed convergence theory.

\subsection{Mathematical preliminaries}\label{subsec:math-prelim}
We consider an open connected domain $\Ome\subset \mathbb{R}^d\ (d=2,3)$ 
with Lipschitz boundary $\p\Ome$.
The $L^2(\Ome)$ norm and inner product will be denoted by $\|\cdot\|$ 
and $(\cdot,\cdot)$, and $L^2_0(\Omega)$ denotes the zero mean subspace of 
$L^2(\Omega)$.  Throughout this paper, it is understood by context whether a particular 
space is scalar or vector valued, and we do not distinguish notation.

For the natural NSE velocity and pressure spaces, we denote
$\bX \coloneqq \bH^1_0(\Omega)$ and $Q \coloneqq L^2_0(\Omega)$.
In the space $X$, the Poincare inequality is known to hold: 
There exists $\lambda>0$, dependent only on $| \Omega| $, such that for every $v\in X$,
$\| v\| \le \lambda \| \nabla v\|.$
The dual space of $X$ will be denoted by $X'$, with norm $\| \cdot \|_{-1}$.  We use the notation $\langle \cdot, \cdot \rangle$ to denote the dual pairing of functions in $X$ and $X'$.

Define the skew-symmetric, trilinear operator $b^{*}:\bX\times \bX \times \bX \rightarrow \mathbb{R}$   by
\[
b^{*}(\bu,\bv,\bw) \coloneqq \frac12 (\bu\cdot\nabla \bv,\bw) - \frac12 (\bu \cdot\nabla \bw,\bv),
\]
and recall, from e.g. \cite{GR86}, that there exists $M$ depending only on $\Omega$ such that
\begin{equation}
| b^{*}(\bu,\bv,\bw) | \le M \| \nabla u \| \| \nabla v \| \| \nabla w\|, \label{eqn:bbound}
\end{equation}
for every $u,v,w\in X$.

Let $\tau_h$ be a  conforming, shape-regular, and simplicial triangulation of $\Ome$ with
maximum element diameter $h$.  Denote by $P_k$ the space of degree $k$ globally continuous piecewise polynomials with respect to $\tau_h$, and $P_k^{disc}$ the space of degree $k$ piecewise polynomials on $\tau_h$ that can be discontinuous across elements.

Throughout the paper, we consider only discrete velocity-pressure spaces $(\bX_h,\ Q_h)\subset(X,Q)$ that satisfy the LBB condition: there exists a constant $\beta$, independent of $h$, satisfying
\begin{align*}
\inf_{q\in Q_h} \sup_{\bv\in \bX_h} \frac{(\nab \cdot \bv,q)}{\|q\| \|\nab \bv\| }\ge \beta>0.
\end{align*}
Common examples of such elements include $(P_2,P_1)$ Taylor-Hood elements, and 
divergence-free $(P_k,P_{k-1}^{disc})$ Scott-Vogelius (SV) elements on meshes with 
particular structure \cite{arnold:qin:scott:vogelius:2D,zhang:scott:vogelius:3D}, 
and see \cite{FalkNeilan13,GuzmanNeilan13} for other stable and divergence-free elements.  
We denote the discretely divergence free 
velocity space by
 \[
V_h \coloneqq \{ v \in X_h,\ (\nabla \cdot v,q)=0 \ \forall q\in Q_h \}.
\]

\subsection{Discrete Navier-Stokes equations}\label{subsec:dNSE}
We can now state the discrete steady NSE problem as follows:  
Find $(u,p)\in (X_h,Q_h)$ satisfying for all $(v,q)\in (X_h,Q_h)$,
\begin{eqnarray}
b^*(u,u,v) - (p,\nabla \cdot v) + \nu(\nabla u,\nabla v) 
& = & \langle f,v \rangle, \label{eqn:dnse1} \\
(\nabla \cdot u,q) & = & 0. \label{eqn:dnse2}
\end{eqnarray}
As shown in \cite{GR86,Laytonbook,temam}, solutions to
\eqref{eqn:dnse1}-\eqref{eqn:dnse2} exist and satisfy \begin{equation}
\label{dsolnbd}
\| \nabla u \| \le \nu^{-1} \| f \|_{-1}.\end{equation}
Define the data-dependent constant $\kappa := M\nu^{-2} \| f \|_{-1}$.  
If the data satisfy the condition $\kappa<1$, then 
the system \eqref{eqn:dnse1}-\eqref{eqn:dnse2} is well-posed with a unique solution pair 
$(u,p)$ \cite{GR86}.
We will assume throughout this paper that  $\kappa<1$, and refer to this as the 
{\it small data condition}. 

It will be notationally convenient to also consider the $V_h$ formulation of 
\eqref{eqn:dnse1}-\eqref{eqn:dnse2}: Find $u \in V_h$ satisfying for all $v\in V_h$
\begin{equation}\label{eqn:dnse3}
b^*(u,u,v)+ \nu(\nabla u,\nabla v)   =  \langle f,v \rangle.
\end{equation} 
The equivalence of \eqref{eqn:dnse3} to \eqref{eqn:dnse1}-\eqref{eqn:dnse2} 
follows from the inf-sup condition \cite{Laytonbook}.

\begin{remark}
The accuracy of the discrete solution can be improved with the use of 
grad-div stabilization in the discrete NSE system, i.e. by
adding $\gamma (\nabla \cdot u,\nabla \cdot v)$ to the momentum equation with 
$\gamma>0$ \cite{JJLR13,OR04}.  To simplify the presentation, we omit this important term, as all the analysis to follow will hold if grad-div is added to the system.
\end{remark}

The Picard iteration, stated as follows, is a common approach to solving
\eqref{eqn:dnse1}-\eqref{eqn:dnse2}.
\begin{alg}[Picard iteration for steady NSE]\label{alg:usualPicard} \ \\
Step 1: Choose $u_0\in X_h.$\\
Step $k$: Find $(u_k,p_k)\in (X_h,Q_h)$ satisfying for all $(v,q)\in (X_h,Q_h)$,
\begin{eqnarray} \label{eqn:pnse1}
b^*(u_{k-1},u_k,v) - (p_k,\nabla \cdot v) + \nu(\nabla u_k,\nabla v)  & = &  \langle f,v \rangle, \\
(\nabla \cdot u_k,q) & = & 0. \label{eqn:pnse2}
\end{eqnarray}
\end{alg}

This algorithm converges with contraction ratio $\kappa$ for 
{any} initial guess, provided $\kappa<1$ (see \cite{GR86} for a standard proof).
We note that the equivalent $V_h$ formulation of Step $k$ of the Picard iteration can be 
written as:  Find $u_k \in V_h$ satisfying for all $v\in V_h$
\begin{equation}\label{eqn:pnse3}
b^*(u_{k-1},u_k,v)+ \nu(\nabla u_k,\nabla v)  =   \langle f,v \rangle.
\end{equation} 

\subsection{Properties of the Picard solution operator for steady NSE}
\label{subsec:prop-Pic-NSE}
In order to analyze the effect of Anderson acceleration on the steady NSE Picard 
iteration, we next define a solution operator for the Picard linearization of the NSE
from \eqref{eqn:pnse3}.

\begin{definition}\label{def:picard-op}
Define the Picard solution operator $ G: V_h \rightarrow V_h$ as follows.  Given $w\in V_h$, $G(w)\in V_h$ satisfies 
\begin{equation} \label{eqn:Gop}
b^*(w,G(w),v)  + \nu(\nabla G(w),\nabla v)   =   \langle f,v \rangle \ \forall v\in V_h.
\end{equation}
\end{definition}
By this definition of $G$, Step $k$ of the Picard iteration \eqref{eqn:pnse3} 
for the steady NSE can be written simply as: set $u_{k} = G(u_{k-1})$.
The problem \eqref{eqn:Gop} is linear, and since $f\in X'$ is assumed, Lax-Milgram theory 
can easily be applied to show
that \eqref{eqn:Gop} is well-posed and thus that the solution operator $G$ is well-defined.  
By taking $v=G(w)$, the trilinear term vanishes, leaving
$
\nu \| \nabla G(w) \|^2  =  \langle f,G(w) \rangle \le \| f\|_{-1} \| \nabla G(w) \|,
$
and thus we have that for any $w\in V_h$, 
\begin{equation}
\| \nabla G(w) \| \le \nu^{-1} \| f \|_{-1}. \label{eqn:Gbound}
\end{equation}

We now prove that $G$ is Lipschitz continuously (Frechet) differentiable, and a 
contractive operator with contraction ratio $\kappa$.

\begin{lemma}\label{lem:Gprop}
The operator $G$ is Lipschitz continuously (Frechet) differentiable, and for any 
$w\in V_h$ satisfies $\| \nabla G'(w) \| \le \kappa$.
\end{lemma}
\begin{remark}
By standard fixed point theory, Lemma \ref{lem:Gprop} implies convergence of the Picard 
algorithm, Algorithm \ref{alg:usualPicard}, under the
small data condition $\kappa<1$.  
Moreover, the convergence is global since the result will hold for any initial guess.
\end{remark}

\begin{proof} For $w,\ h \in V_h$, consider equations for $G(w)$ and $G(w+h)$ defined 
by \eqref{eqn:Gop}:
\begin{eqnarray*}
b^*(w,G(w),v)  + \nu(\nabla G(w),\nabla v)   &=&  \langle f,v\rangle \ 
\forall v\in V_h, \\
b^*(w+h,G(w+h),v)  + \nu(\nabla G(w+h),\nabla v)   &=&  \langle f,v\rangle \ 
\forall v\in V_h.
\end{eqnarray*}
Subtracting yields
\begin{eqnarray}
b^*(w+h,G(w+h)-G(w),v) + b^*(h,G(w),v)  + \nu(\nabla (G(w+h)-G(w)),\nabla v) = 0. 
\label{Gdiff}
\end{eqnarray}
Now setting $v=G(w+h)-G(w)$ vanishes the first nonlinear term, and produces 
\begin{eqnarray*}
\nu \| \nabla( G(w+h)-G(w) ) \|^2 & \le & | b^*(h,G(w),G(w+h)-G(w)) | \\
& \le & M \| \nabla h \| \nabla G(w) \|  \| \nabla( G(w+h)-G(w) ) \| \\
& \le & \nu^{-1}M\|f\|_{-1} \| \nabla h \|  \| \nabla( G(w+h)-G(w) ) \|,
\end{eqnarray*}
thanks to \eqref{eqn:bbound} and \eqref{eqn:Gbound}.  This reduces immediately to
\begin{eqnarray}\label{eqn:lip}
\| \nabla( G(w+h)-G(w) ) \| & \le \kappa \| \nabla h\|,
\end{eqnarray}
which proves $G$ is Lipschitz continuous and contractive with contraction ratio $\kappa$.

Next we show the $G$ is Frechet differentiable. First define for a given $w\in V_h$ 
an operator $A_w: V_h \rightarrow V_h$ such that for all $h \in V_h$
\begin{equation} 
\label{defA}
b^*(h, G(w), v) + b^*(w,A_w(h),v) + \nu(\nabla A_w(h), \nabla v) = 0  \ \ 
\forall v\in V_h.
\end{equation}
Using properties for $G$ and $b^*$ established above together with  Lax-Milgram theory
it is easily verified the this linear problem is well-posed and thus $A_w$ is 
well-defined. 

Subtracting \eqref{defA} from \eqref{Gdiff} provides 
\begin{align*}
b^*(w,G(w+h)-G(w)-A_w(h),v)   + & \nu(\nabla (G(w+h)-G(w)-A_w(h)),\nabla v)  \\
&=  -b^*(h,G(w+h)-G(w),v) \\
&\leq M \| \nabla h\| \| \nabla (G(w+h) - G(w))\| \|\nabla v\| \\
&\leq \kappa M\|\nabla h\|^2 \|\nabla v\|,
\end{align*}
for all $v\in V_h$ thanks to \eqref{eqn:bbound} for the first inequality and 
\eqref{eqn:lip} for the second.  This proves that $G$ is Frechet differentiable
at $w$.  From \eqref{eqn:lip} and noting $w\in V_h$ is arbitrary
establishes the result.
\end{proof}

\section{The Anderson-accelerated Picard iteration for NSE}\label{sec:AAP-NSE}
In this section, we define, analyze and test an Anderson-accelerated Picard iteration 
for the steady incompressible NSE.   Although usual Picard, Algorithm 
\ref{alg:usualPicard}, is stable and globally convergent under a small data condition, its
convergence rate can be sufficiently slow that it may fail in practice.
The goal of  combining the Picard iteration with Anderson
acceleration is to improve convergence properties without introducing significant
extra cost.  

We define the Anderson-accelerated Picard iteration for the incompressible steady NSE 
(AAPINSE) as Algorithm \ref{alg:anderson} with $G$ given by \eqref{eqn:Gop}, the 
solution operator for the Picard linearized NSE.  
We note that optimization step of Algorithm \ref{alg:anderson} is negligible in 
computational cost compared to the linear solve associated with applying the $G$ operator.
Hence for each iteration,  this method has nearly the same computational expense as 
usual Picard.

Combining Theorem \ref{thm:conv1} with Lemma \ref{lem:Gprop} establishes local 
convergence of the AAPINSE under the assumption of uniformly bounded optimization 
parameters and a good initial guess.  We prove next for AAPINSE 
that the acceleration does in fact improve the convergence rate of the 
fixed point iteration based on the improvement given by the optimization.  We provide 
results below for the cases of $m=1$ and $m=2$.  We were unable to find an easily 
digestible proof for general $m$, but expect extension to greater values of $m$ 
will follow along similar lines.

\begin{theorem}[Improved convergence of the AAPINSE residual with $m=1$] 
\label{thm:m1}
Suppose $0 < |\alpha_{k-1}^k| < \bar \alpha$ for some fixed $\bar \alpha$.
Then on any step where $\alpha_{k-2}^k \ne 0$, 
the $m=1$ Anderson accelerated Picard iterates satisfy
\begin{equation}\label{eqn:ck-res}
\| \nabla ( G( u_{k} ) - u_k) \| \le \kappa \|\nabla (G(u_{k-1}) - u_{k-1}) \|
\left(\theta_k + C_0 \nr{\grad (G(u_{k-2}) - u_{k-2})} \right),
\end{equation}
with $C_0 = \nu^{-1}M \bar\alpha/(1-\kappa)^2$ 
and where $0 \le \theta_k \le  \theta$ for some fixed $\theta < 1$ 
represents the improvement from the optimization at 
Step $k$ and satisfies \eqref{eqn:n3-002}.
\end{theorem}
On any step where $\alpha_{k-2}^k = 0$, meaning $u_k = G(u_{k-1})$ (the standard Picard 
iteration) it holds that $\theta = 1$ and 
$\|G(u_k) - u_k\| \le \kappa \| G(u_{k-1}) - u_{k-1}\|$.  
Assuming $\theta_k < \theta$ for some $\theta < 1$, Theorem \eqref{thm:m1} 
yields an improved convergence rate as $k$ increases,
based on the success of the optimization problem.  
Unlike Theorem \ref{thm:reduc}, the improved
convergence rate is only local; however, the assumptions on the optimization 
coefficients are significantly weaker.

\begin{proof}
Define $e_k, \te_k$ and $w_k$ by \eqref{eqn:notation}.
The structure of the proof is first to establish two key inequalities that bound the 
error by the residual
\begin{align}
\| \nabla \te_{k} \| &\le \kappa \| \nabla e_{k-1} \|,
\label{eqn:crit1} \\
\| \nabla e_k     \| &\le \f{1}{1-\kappa} \| \nabla w_{k-1}\|,
\label{eqn:crit2}
\end{align}
and then to use these for the NSE-specific main result. The first inequality 
\eqref{eqn:crit1} follows directly from \eqref{eqn:lip}. 
The second follows from the decomposition
$e_k = (u_k - \tu_{k}) + (\tu_k - u_{k-1}) = -\alpha_{k-2}^k \te_k + w_{k-1}.$
Using \eqref{eqn:crit1} we have
\begin{equation}\label{eqn:c2e1}
\| \nabla e_{k}\| \le \kappa |\alpha_{k-2}^k| \|\nabla e_{k-1}\|+ \| \nabla w_{k-1}\|.
\end{equation}
The first term on the right of \eqref{eqn:c2e1} can be controlled by the ``backwards'' 
inequality 
\begin{equation}
\| \nabla e_{k-1} \| \le \frac{1}{(1-\kappa) | \alpha^{k}_{k-2} | } \| \nabla w_{k-1} \|,
\label{eqn:c2e2}
\end{equation}
which follows from the closed form expression for $\alpha_{k-2}^{k}$ for $m=1$.
It is based on the contribution $u_{k}$ has from $\tu_{k-1}$, 
and requires the assumption $\alpha_{k-2}^k$ is nonzero.
For $m=1$ the optimization Step $k [b.]$ of Algorithm \ref{alg:anderson} can be 
written as 
$
\alpha_{k-2}^{k} = \argmin{}_{\alpha \in \mathbb{R}} \left\| \grad 
\left( w_{k-1}  + \alpha \left( w_{k-2} - w_{k-1} \right) \right) \right\|,
$ 
from which exploiting the Hilbert space structure 
\[
\alpha^{k}_{k-2} \| \nabla\left( w_{k-1} - w_{k-2} \right) \|^2 
= \left(\nabla  w_{k-1},\nabla \left( w_{k-1} - w_{k-2} \right) \right).
\]
Applying Cauchy-Schwarz on the right reduces this to
$
\| \nabla\left( w_{k-1} - w_{k-2} \right) \|
\le \frac{1}{| \alpha^{k}_{k-2}| }\| \nabla w_{k-1} \|.
$
By the identity $w_{k-1} - w_{k-2} = \te_k - \te_{k-1}$ and the triangle inequality 
\begin{equation}\label{eqn:c2e3}
(1-\kappa)\| \grad e_{k-1}\| \le
\| \nabla e_{k-1} \| - \| \nabla \te_{k} \|
\le 
\| \nabla\left( \tilde e_{k} - e_{k-1}  \right) \|
\le
\frac{1}{| \alpha^{k}_{k-2} | }\| \nabla w_{k-1} \|,
\end{equation}
where the first inequality follows from \eqref{eqn:crit1}.  
Comparing the first and last terms of \eqref{eqn:c2e3} verifies \eqref{eqn:c2e2}, 
and applying \eqref{eqn:c2e2} to \eqref{eqn:c2e1} validates \eqref{eqn:crit2}.

To establish the main result of the theorem, we make use of the two following identities 
which follow from Algorithm \ref{alg:anderson} and 
$u_k =\alpha_{k-1}^k \tilde{u}_k + \alpha_{k-2}^k \tilde u_{k-1}$
\begin{eqnarray}
 \alpha^{k}_{k-1} \tilde e_k &=&  u_k - \tilde u_{k-1} , 
\label{eqn:e1} \\ 
 e_k + \alpha_{k-2}^k e_{k-1} &=& \alpha_{k-1}^k w_{k-1} + \alpha_{k-2}^k w_{k-2}. 
\label{eqn:e3}
\end{eqnarray}
From $\tilde u_{k+1}=G(u_k)$, and \eqref{eqn:Gop}, we have for 
$j \ge 1$ 
\begin{eqnarray} 
\nu(\nabla \tilde u_{j+1}, \nabla v) + b^*(u_{j},\tilde u_{j+1}, v) & = & 
\langle f,v\rangle ~\text{ for all } v \in V_h. \label{eqn:pa0}
\end{eqnarray}

Adding
$\alpha_{k-1}^k$ times \eqref{eqn:pa0} with $j=k-1$ to 
$\alpha_{k-2}^k$ times \eqref{eqn:pa0} with $j = k-2$ and applying the definition of 
$u_k$ together with $\alpha_{k-1}^k + \alpha_{k-2}^k = 1$ produces the equation for $u_k$:
\begin{equation} \label{eqn:pa3}
\nu (\nabla u_k, \nabla v) + b^*(u_{k-1}, u_k,v) - b^*(e_{k-1},\alpha_{k-2}^k \tilde u_{k-1}, v) = \langle f, v \rangle.
\end{equation}
Subtracting \eqref{eqn:pa3} from \eqref{eqn:pa0}, with $j = k$, obtain
\[
\nu (\nabla (\tilde u_{k+1} - u_k), \nabla v) + b^*(u_{k}, \tilde u_{k+1} - u_k, v) 
+ b^*(e_k ,u_k,v) + \alpha_{k-2}^k b^*(e_{k-1}, \tilde u_{k-1}, v) =  0,
\]
which by \eqref{eqn:e1} is equivalent to
\begin{equation}
\nu (\nabla w_k, \nabla v) + b^*(u_{k}, w_k, v) +  
b^*(e_k + \alpha_{k-2}^k e_{k-1}, \tilde u_{k-1},v) + b^*(e_k, \alpha_{k-1}^k \tilde e_k,v)
= 0. \label{eqn:pa3a}
\end{equation}
Choosing $v = w_k$ in \eqref{eqn:pa3a} vanishes the second term.
Applying \eqref{eqn:bbound} and \eqref{eqn:e3} yields
\[
\| \nabla w_k\| \le 
M\nu^{-1} \left(
\|\nabla ( \alpha_{k-1}^k w_{k-1} + \alpha_{k-2}^k w_{k-2}) \| \| \nabla \tilde u_{k-1}\| 
+ \kappa | \alpha_{k-1}^k|  \| \nabla e_{k} \|  \| \nabla e_{k-1} \| \right). 
\]
Finally, applying $\| \nabla \tilde u_{k-1}\| \le \nu^{-1}\|f\|_{-1}$ from 
\eqref{eqn:Gbound} 
together with \eqref{eqn:n3-002} and \eqref{eqn:crit2} we have
\begin{align*}
\| \nabla w_k\| & \le  \kappa \theta_k \|\nabla w_{k-1} \| 
  + \kappa \nu^{-1} M |\alpha_{k-1}^k|\| \nabla e_k\| \|\nabla e_{k-1}\| 
\\
& \le  \kappa \|\nabla w_{k-1} \| \left( \theta + 
\frac{\nu^{-1} M |\alpha_{k-1}^k|}{(1-\kappa)^2} \|\nabla w_{k-2}\| 
\right).
\end{align*}
\end{proof}
Together with the contraction of the underlying fixed-point iteration, 
Theorem \ref{thm:m1} establishes convergence of the residual to zero after the 
first iterate that satisfies $\|\nabla w_{k-2}\| < (1- \kappa\theta)/(\kappa C_0)$;
and, contraction at a faster rate than the fixed-point iteration once 
$\|\nabla w_{k-2}\| < (1 - \theta)/C_0$.
The underlying assumption that the gain from the optimization step is bounded away from 
unity by some fixed $\theta$ for bounded coefficients on steps for which there is a  
contribution to $u_k$ from $\tu_{k-1}$ is a reasonable characterization of 
conditions under which the algorithm should be expected to succeed.

Next, we establish improved convergence of AAPINSE for the case $m=2$.  
The proof strategy is analogous to the $m=1$ case, but with additional technical
details arising from the additional parameter in the optimization step.
We provide the $m=2$ proof as an indication that the extension to greater $m$ would 
follow the same essential idea.

\begin{theorem}[Improved convergence of the AAPINSE residual with $m =2$]\label{thm:m2}
Suppose the coefficients $|\alpha_{j}^{k+1}|$ are bounded, $j = {k-2,k-1,k}$, 
the coefficient corresponding to the latest fixed-point iterate satisfies
$|\alpha_{k}^{k+1}| > \breve \alpha > 0$ and $\alpha_{k}^{k+1} > \alpha_{k-2}^{k+1}$.
Then on any step where at least one of  $\alpha_{k-2}^{k+1}$ or $\alpha_{k-1}^{k+1}$ 
is nonzero the $m = 2$ Anderson accelerated Picard iteration satisfies
\[
\|\nabla (\tilde u_{k+2}  - u_{k+1})\| \le \kappa\theta_{k+1}\| \nabla (\tilde u_{k+1} - u_k)\| + \mathcal{O}(\| \nabla (\tilde u_{k-1}- u_{k-2})\|^2) ,
\]
where $0\le \theta_{k+1}\le \theta$ for some fixed $\theta < 1$ 
satisfies \eqref{eqn:n3-002}.
\end{theorem}

The proof follows the same general strategy as the $m=1$ case, and again establishes
local convergence of the algorithm (with mild assumptions on the coefficients) after
the first iterate where $\norm{\grad w_{k-2}}$ is small enough; 
and, with an improved rate 
when the accelerated solution is other than the fixed-point iterate.
We precede the proof with a technical lemma to establish four key inequalities
which bound the difference between accelerated iterates by the latest three residuals. 
As this is a general result (not NSE-specific), it is posed in the same notation as 
\S \ref{sec:anderson}.
\begin{lemma}\label{lem:m2}
Let the sequences $\{u_k\}$ and $\{\tu_k\}$ be given by Algorithm \ref{alg:anderson}
with $m=2$, and define $e_k, \te_k$ and $w_k$ by \eqref{eqn:notation}.
Let $G: X \goto X$ satisfy Assumption \ref{assume:G} with constant $r< 1$ 
where $X$ is a Hilbert space
with norm $\ns{\cdot}$ induced by inner product $(\,\cdot ,\cdot \, )_*$.
Then the following hold for $k > 1$.
\begin{align} 
\label{eqn:ekalpha}
|\alpha_k^{k+1}|\ns{e_{k}} & \le \frac{1}{(1-r)}
  \left(|1-\alpha_{k-2}^{k+1}|\ns{w_{k-1}} 
  + |\alpha_{k-2}^{k+1}|\ns{w_{k-2}}\right)
\\ \label{eqn:ekbeta}
|1-\alpha_k^{k+1}|\ns{e_{k}} & \le \frac{1}{(1-r)}
  \left(|1-\alpha_{k}^{k+1}|\ns{w_{k-1}} 
   + (1+|\alpha_{k}^{k+1}|)\ns{w_{k}}\right)
\\ \label{eqn:ek1alpha}
|\alpha_{k-2}^{k+1}|\ns{e_{k-1}} & \le \frac{1}{(1-r)}
  \left(|1-\alpha_{k}^{k+1}|\ns{w_{k-1}} 
  + |\alpha_{k}^{k+1}|\ns{w_{k}}\right)
\\ \label{eqn:ek1beta}
|1-\alpha_{k-2}^{k+1}|\ns{e_{k-1}} & \le \frac{1}{(1-r)}
  \left(|1-\alpha_{k-2}^{k+1}|\ns{w_{k-1}} 
  + (1+|\alpha_{k-2}^{k+1}|)\ns{w_{k-2}}\right)
\end{align}
\end{lemma}

\begin{proof} 
Without confusion, denote $\alpha_j^{k+1}$ by $\alpha_j$, for $j = \{k-2, k-1,k\}$.
First, by  \ref{assume:G} and the triangle inequality we have
\begin{align} 
\label{eqn:lem2-e0}
(1-r) \ns{e_n} \le \ns{e_n} - \ns{ \te_{n+1} } \le \ns{\te_{n+1} - e_n} 
= \ns{ w_{n} - w_{n-1}}.
\end{align}
To derive \eqref{eqn:ekalpha} and \eqref{eqn:ek1beta}, write the
Step $k [b.]$ minimization problem of Algorithm \ref{alg:anderson} in the equivalent form:
Find $(\alpha_k,\beta_0)$ that minimize 
\[
\ns{ \left( \alpha_{k} (w_{k}-w_{k-1}) + \beta_{0}(w_{k-1}-w_{k-2}) + 
w_{k-2} \right) }^2,
\]
with $\beta_0 = \alpha_k + \alpha_{k-1}$ (so from 
$\alpha_k + \alpha_{k-1} + \alpha_{k-1}=1$ we have $1-\beta_0 = \alpha_{k-2}$).
Exploiting the Hilbert space structure, the critical points $\alpha_k$ and $\beta_0$ 
are the solutions of 
\begin{align}\label{eqn:lem2-e1}
\alpha_{k}\ns{w_k - w_{k-1}}^2 &= 
 -(w_k - w_{k-1}, \beta_0 w_{k-1} + (1-\beta_0)w_{k-2})_*,
\\ \label{eqn:lem2-e2}
\beta_0\ns{w_{k-1} - w_{k-2}}^2 &= 
 -(w_{k-1} - w_{k-2}, \alpha_k (w_k -w_{k-1}) + w_{k-2})_*.
\end{align}
Applying Cauchy-Schwarz and triangle inequalities together to \eqref{eqn:lem2-e1}  
yields
\begin{align}\label{eqn:lem2-e3}
|\alpha_k| \ns{w_k - w_{k-1}} \le |1-\alpha_{k-2}|\ns{w_{k-1}} 
  + |\alpha_{k-2}|\ns{w_{k-2}}.
\end{align}
Applying the same estimates together with \eqref{eqn:lem2-e3} to \eqref{eqn:lem2-e2} 
yields
\begin{align}\label{eqn:lem2-e4}
|\beta_0| \ns{w_{k-1} - w_{k-2}} \le |1-\alpha_{k-2}|\ns{w_{k-1}} 
+ (1+|\alpha_{k-2}|)\ns{w_{k-2}}.
\end{align}
Combining \eqref{eqn:lem2-e0} with \eqref{eqn:lem2-e3} (respectively \eqref{eqn:lem2-e4})
yields \eqref{eqn:ekalpha} (respectively \eqref{eqn:ek1beta}).

Following the same process with the minimization problem written in the equivalent form:
Find $(\beta_1,\alpha_{k-2})$ that minimize 
\[
\ns{ \left( w_k + \beta_{1} (w_{k-1}-w_{k}) + \alpha_{k-2}(w_{k-2}-w_{k-1} 
 \right) }^2,
\]
with $\beta_1 = \alpha_{k-1} + \alpha_{k-2}$ (which implies $1-\beta_1 = \alpha_k$)
establishes \eqref{eqn:ekbeta} and \eqref{eqn:ek1alpha}.
\end{proof}
The purpose of the four estimates \eqref{eqn:ekalpha}-\eqref{eqn:ek1beta}
is to bound the terms $\norm{\grad e_k}$ and $\norm{\grad e_{k-1}}$ where they appear
in the following estimates by $\norm{\grad w_k}, \norm{\grad w_{k-1}}$ 
and $\norm{\grad w_{k-2}}$, without introducing optimization coefficients other than
$\alpha_k^{k+1}$ in the denominator. This is important as only $\alpha_k^{k+1}$ is
justifiably bounded away from zero.
We proceed now with the proof of Theorem \ref{thm:m2} applying Lemma \ref{lem:m2} with
$\ns{v} = \norm{\grad v}$ and $r = \kappa$.
\begin{proof}[Proof of Theorem \ref{thm:m2}]
Recall the solution from Step $k$ is defined as
$u_{k+1} = \alpha_k \tilde u_{k+1} + \alpha_{k-1} \tilde u_{k} + \alpha_{k-2} 
\tilde u_{k-1}$, with $\alpha_j^{k+1}$ denoted $\alpha_j$ for $j = \{k-2,k-1,k\}$.
From the problem definition \eqref{eqn:pnse3}, 
the following equation holds for $n = \{k-2, k-1, k,k+1\}$ 
\begin{align}
\nu( \nabla  \tilde u_{n+1}, \nabla v) + b^*(u_{n}, \tilde u_{n+1}, v) 
= \langle f,v \rangle,
\label{pa00}
\end{align}
thus as in \eqref{eqn:pa3} we have 
\[
\nu(\nabla u_{k+1} , \nabla v) + \sum_{j= k-2}^{k} \alpha_{j} 
b^*(u_j, \tilde u_{j+1}, v) = \langle f,v\rangle.
\]
Subtracting the above equation from \eqref{pa00} with $n = k+1$ yields
\begin{equation}
\label{eqn:l1}
\nu(\nabla (\tilde u_{k+2} - u_{k+1}), \nabla v) + b^*(u_{k+1} , \tilde u_{k+2} - u_{k+1}, v) + b^*(u_{k+1}, u_{k+1},v) 
- \sum\limits_{j= k-2}^{k} \alpha_{j} b^*(u_j, \tilde u_{j+1}, v)  = 0.
\end{equation}
Next, rewrite the last two terms on the left hand side in terms of $e_k$, $\tilde e_k$
and $u_k^\alpha$ given by \eqref{eqn:ua_k}. 
\begin{align*}
&b^*(u_{k+1}, u_{k+1},v) - \sum_{j=k-2}^k \alpha_j b^*(u_j, \tilde u_{j+1}, v) 
\\
&=  b^*\left(u_{k+1}-u_k^\alpha , \tilde u_{k-1},v\right)
+ b^*\left(u_k^\alpha, \tilde u_{k-1},v\right) 
+ b^*(u_{k+1}, u_{k+1} - \tilde u_{k-1},v) 
- \sum\limits_{j= k-2}^{k} \alpha_{j} b^*(u_j, \tilde u_{j+1}, v)
\\
&=  b^*\left(u_{k+1}\!-u_k^\alpha, \tilde u_{k-1},v\right)
+ b^*(u_{k+1}, u_{k+1} \!- \tilde u_{k-1},v) 
\!- b^*(u_k,\alpha_k (\tilde e_{k+1} + \tilde e_k),v) 
\!-   b^*(u_{k-1}, \alpha_{k-1}\tilde e_k,v).
\end{align*}
Now using the identity
$u_{k+1} - \tilde u_{k-1} 
= \alpha_{k} \tilde e_{k+1} + (\alpha_k + \alpha_{k-1} ) \tilde e_{k},$
produces
\begin{align*}
& b^*(u_{k+1}, u_{k+1},v) - \sum_{j= k-2}^{k} \alpha_{j} 
b^*(u_j, \tilde u_{j+1}, v) 
\\ &
=  b^*\left(u_{k+1}-u_k^\alpha, \tilde u_{k-1},v\right) 
 + b^*(e_{k+1}, \alpha_k\tilde e_{k+1} 
+ (\alpha_k + \alpha_{k-1})\tilde e_{k},v)
+ b^*(e_k, \alpha_{k-1} \tilde e_k,v),
\end{align*}
and replacing $e_{k+1}$ by
\[e_{k+1} = (u_{k+1} - \tilde u_{k+1}) + (\tilde u_{k+1} - u_k) = 
- (\alpha_{k-1} + \alpha_{k-2}) \tilde e_{k+1} - \alpha_{k-2}\tilde e_k + (\tilde u_{k+1} - u_k),
\] 
gives
\begin{align*}
& b^*(u_{k+1}, u_{k+1},v) - \sum_{j= k-2}^{k} \alpha_{j} b^*(u_j, \tilde u_{j+1}, v) 
\\ &
= b^*\left(u_{k+1}-u_k^\alpha, \tilde u_{k-1},v\right) 
 - b^*( (\alpha_{k-1} + \alpha_{k-2}) \tilde e_{k+1} 
+ \alpha_{k-2}\tilde e_k, \alpha_k\tilde e_{k+1} 
\\ &
+ (\alpha_k + \alpha_{k-1})\tilde e_{k},v) 
 + b^*(\tilde u_{k+1} - u_k,\alpha_k\tilde e_{k+1} 
+ (\alpha_k + \alpha_{k-1})\tilde e_{k}, v)
+ b^*(e_k, \alpha_{k-1} \tilde e_k,v) .
\end{align*}
Thus, \eqref{eqn:l1} can be written as
\begin{align}
&\nu(\nabla w_{k+1}, \nabla v) 
+ b^*(u_{k+1} , w_{k+1}, v) 
+ b^*\left(u_{k+1}-u_k^\alpha, \tilde u_{k-1},v\right) 
\nonumber \\ &
- b^*( (\alpha_{k-1} + \alpha_{k-2}) \tilde e_{k+1} 
+ \alpha_{k-2}\tilde e_k, \alpha_k\tilde e_{k+1} 
+ (\alpha_k + \alpha_{k-1})\tilde e_{k},v) 
\nonumber \\ &
+ b^*(w_k,\alpha_k\tilde e_{k+1} 
+ (\alpha_k + \alpha_{k-1})\tilde e_{k}, v)
+ b^*(e_k, \alpha_{k-1} \tilde e_k,v)  = 0.
\label{eqn:l01}
\end{align}

Next, setting  $v = w_{k+1}$ in \eqref{eqn:l01} yields
\begin{align}
\nu \| \nabla w_{k+1}\|^2 
&=  - b^*\left(u_{k+1}-u_k^\alpha, \tilde u_{k-1},w_{k+1} \right)
\nonumber \\
& + b^*\left( (\alpha_{k-1} + \alpha_{k-2}) \tilde e_{k+1} + \alpha_{k-2}\tilde e_k, \alpha_k\tilde e_{k+1} + (\alpha_k + \alpha_{k-1})\tilde e_{k},w_{k+1} \right) 
\nonumber \\
& - b^*(w_k,\alpha_k\tilde e_{k+1} + (\alpha_k + \alpha_{k-1})\tilde e_{k}, w_{k+1})
- b^*(e_k, \alpha_{k-1} \tilde e_k,w_{k+1}),
\label{eqn:l2}
\end{align}
and we proceed to bound the right hand side terms.  For the first term
\begin{align*}
 b^*\left(u_{k+1}-u_k^\alpha, \tilde u_{k-1}, w_{k+1}\right) 
& \le M \norm{ \nabla \left(u_{k+1}-u_k^\alpha\right) } 
\norm{ \nabla \tilde u_{k-1}\| \|\nabla w_{k+1}}
\\ & 
\le \nu^{-1}M\|f\|_{-1} \theta_{k}\norm{\nabla w_k}\norm{\nabla w_{k+1}},
\end{align*}
using \eqref{eqn:ua_k}, \eqref{eqn:n3-002} and 
$\| \nabla \tilde u_{k-1}\| \le \nu^{-1}\|f\|_{-1}$. 
The second term of \eqref{eqn:l2} is majorized via
\begin{align}
\label{eqn:bd02}
&M \|\nabla ( (\alpha_{k-1} + \alpha_{k-2}) \tilde e_{k+1} + \alpha_{k-2}\tilde e_k)\| \| \nabla (\alpha_k\tilde e_{k+1} + (\alpha_k + \alpha_{k-1})\tilde e_{k})\| 
\|\nabla w_{k+1}\| \nonumber\\
&\le  M \kappa^2 \| \nabla w_{k+1}\|
 \left(     
 | 1- \alpha_{k}| | \alpha_k \| \nabla e_k\|^2 +  |1 - \alpha_{k-2}| |\alpha_{k-2}| 
\norm{\nabla e_{k-1}}^2 \right) \nonumber \\
& +M \kappa^2 \norm{\nabla w_{k+1}} 
\left( |\alpha_k||\alpha_{k-2}| + |1 - \alpha_{k}| |1- \alpha_{k-2}|  \right) 
\norm{\nabla e_k\|\|\nabla e_{k-1}}.
\end{align}
Applying \eqref{eqn:ekalpha}-\eqref{eqn:ek1beta} from Lemma \ref{lem:m2}, 
\eqref{eqn:bd02} is controlled by 
\begin{align}\label{eqn:bd02e}
& \f{M \kappa^2}{(1-\kappa)^2} \| \nabla w_{k+1}\|
\nonumber \\ & \times
 \Big(     
 \big( |1-\alpha_{k-2}|\norm{\grad w_{k-1}} + |\alpha_{k-2}|\norm{\grad w_{k-2}}\big)
 \big(|1-\alpha_{k}|\norm{\grad w_{k-1}} + (1+|\alpha_{k}|)\norm{\grad w_{k}}\big)
\nonumber \\ & +
 \big( |1-\alpha_{k}|\norm{\grad w_{k-1}} + |\alpha_{k}|\norm{\grad w_{k}}\big)
 \big(|1-\alpha_{k-2}|\norm{\grad w_{k-1}} + (1+|\alpha_{k-2}|)\norm{\grad w_{k-2}}\big)
 \Big)
\nonumber \\ & + \Big(
 \big( |1-\alpha_{k-2}|\norm{\grad w_{k-1}} + |\alpha_{k-2}|\norm{\grad w_{k-2}}\big)
 \big( |1-\alpha_{k}|\norm{\grad w_{k-1}} + |\alpha_{k}|\norm{\grad w_{k}}\big)
\nonumber \\ & +
 \big(|1-\alpha_{k}|\norm{\grad w_{k-1}} + (1+|\alpha_{k}|)\norm{\grad w_{k}}\big)
 \big(|1-\alpha_{k-2}|\norm{\grad w_{k-1}} + (1+|\alpha_{k-2}|)\norm{\grad w_{k-2}}\big)
\Big).
\end{align}
Using \eqref{eqn:ekalpha} and \eqref{eqn:ek1beta},
the third term on the right hand side of \eqref{eqn:l2} is bounded by 
\begin{align*}
&M\kappa 
\| \nabla w_{k+1} \| \| \nabla w_{k}\| 
\left( |\alpha_k| \| \nabla e_k\| + |1 - \alpha_{k-2}| \|\nabla e_{k-1}\| \right)
\\ & \le 
\frac{M\kappa}{(1-\kappa)} 
\| \nabla w_{k+1} \| \| \nabla w_{k}\| 
\big(2|1-\alpha_{k-2}|\norm{\grad w_{k-1}} + (1+2|\alpha_{k-2}|)\norm{\grad w_{k-2}}\big).
\end{align*}
By the assumption $\alpha_k \ge \alpha_{k-2}$ we have
\[
\alpha_{k-1} = (\alpha_{k-1} + \alpha_{k-2}) - \alpha_{k-2}
             = (1 - \alpha_k) - \alpha_{k-2} \le (1-\alpha_{k-2}) - \alpha_{k-2}.
\]
Using this together with \eqref{eqn:ekalpha},\eqref{eqn:ek1alpha} and \eqref{eqn:ek1beta},
the last term of \eqref{eqn:l2} is controlled by
\begin{align}
\label{eqn:bd04}
&M \kappa \norm{\nabla w_{k+1}}
| \alpha_{k-1}|  \| \nabla e_k\|  \| \nabla e_{k-1}\|
\nonumber \\ & \le
\frac{M \kappa}{(1-\kappa)^2} \norm{\grad w_{k+1}} \frac{1}{|\alpha_k|}
 \big( |1-\alpha_{k-2}|\norm{\grad w_{k-1}} + |\alpha_{k-2}|\norm{\grad w_{k-2}}\big)
\nonumber \\ & \times
 \big( (|1-\alpha_{k}| + |1 - \alpha_{k-2}|)\norm{\grad w_{k-1}} 
  + |\alpha_{k}|\norm{\grad w_{k}} + (1+|\alpha_{k-2}|)\norm{\grad w_{k-2}}\big).
\end{align}

Finally, combining \eqref{eqn:l2}-\eqref{eqn:bd04} yields
\begin{align*}
\norm{\grad w_{k+1}} & \le
\kappa \theta_k \norm{\grad w_k}  
+\frac{M \nu^{-1} \kappa}{(1-\kappa)}\Bigg( \norm{\grad w_k}
\left(c_1 \norm{\grad w_{k-1}} + c_2 \norm{\grad w_{k-2}} \right)
\\ & +
\left(\frac{\kappa}{1-\kappa}
+\frac{1}{\breve \alpha (1-\kappa)}\right)\times
\bigo\left(\norm{\grad w_{k-2}}^2\right) \Bigg)
\\ & = 
\kappa \theta_k \norm{\grad w_k} + \bigo\left(\norm{\grad w_{k-2}}^2\right),
\end{align*}
where all the implicitly defined constants are sums and products of the bounded
$|\alpha_k|, |1-\alpha_k|, |\alpha_{k-2}|$ and $|1-\alpha_{k-2}|$.  The only optimization
coefficient that makes an appearance in a denominator is  $\alpha_{k}^{k+1}$.  
It is a reasonable assumption this coefficient is bounded away from zero as without a 
contribution from the latest fixed-point iterate $\tilde u_{k+1}$, the new solution 
$u_{k+1}$ remains spanned by the same (less one) basis vectors as $u_k$ and should
not yield an improved residual.
\end{proof}

\section{Numerical experiments}\label{sec:numerics}
Here we present numerical experiments to show the improved convergence provided by the 
Anderson acceleration for solving the steady NSE.  
As illustrated below, Anderson acceleration can provide fast 
convergence even when Newton and usual Picard iterations fail.
Our test problems are the 2D and 3D driven cavity, at varying Reynolds numbers.  All computations were done in Matlab with the
authors' codes, and `fminsearch' was used to solve the optimization problems.

\subsection{2D lid driven cavity}\label{subsec:2Dlid}
We test now AAPINSE on the 2D driven cavity, at 
benchmark values of $Re=$1000, 2500, and 5000, and compare results with those of the 
usual Picard and Newton methods.
\begin{figure}[H]
\begin{center}
$Re$=1000 \hspace{1.4in} $Re$=2500 \hspace{1.4in} $Re$=5000  \\
\includegraphics[width = .32\textwidth, height=.3\textwidth,viewport=50 0 470 395, clip]{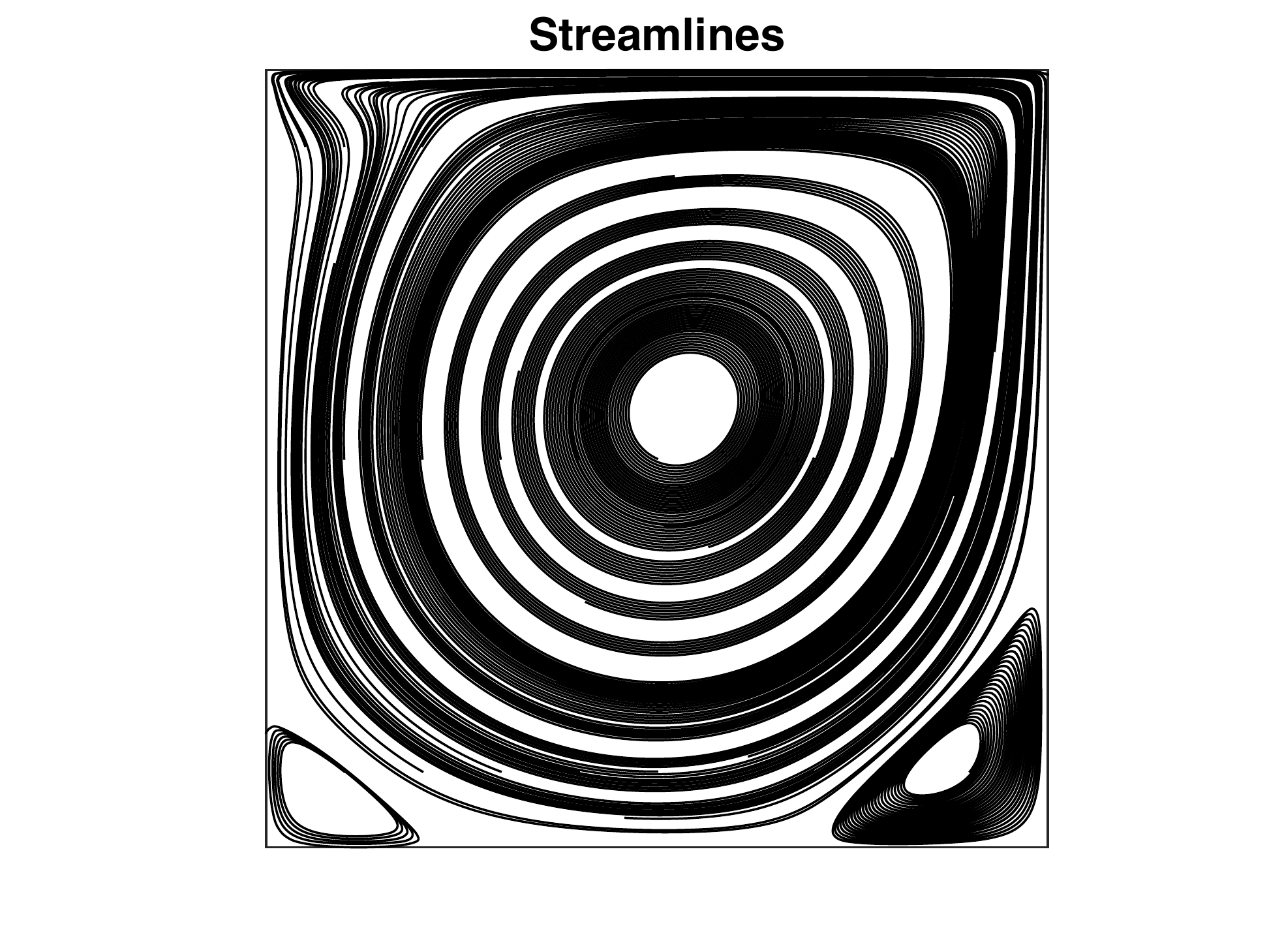}
\includegraphics[width = .32\textwidth, height=.3\textwidth,viewport=50 0 470 395, clip]{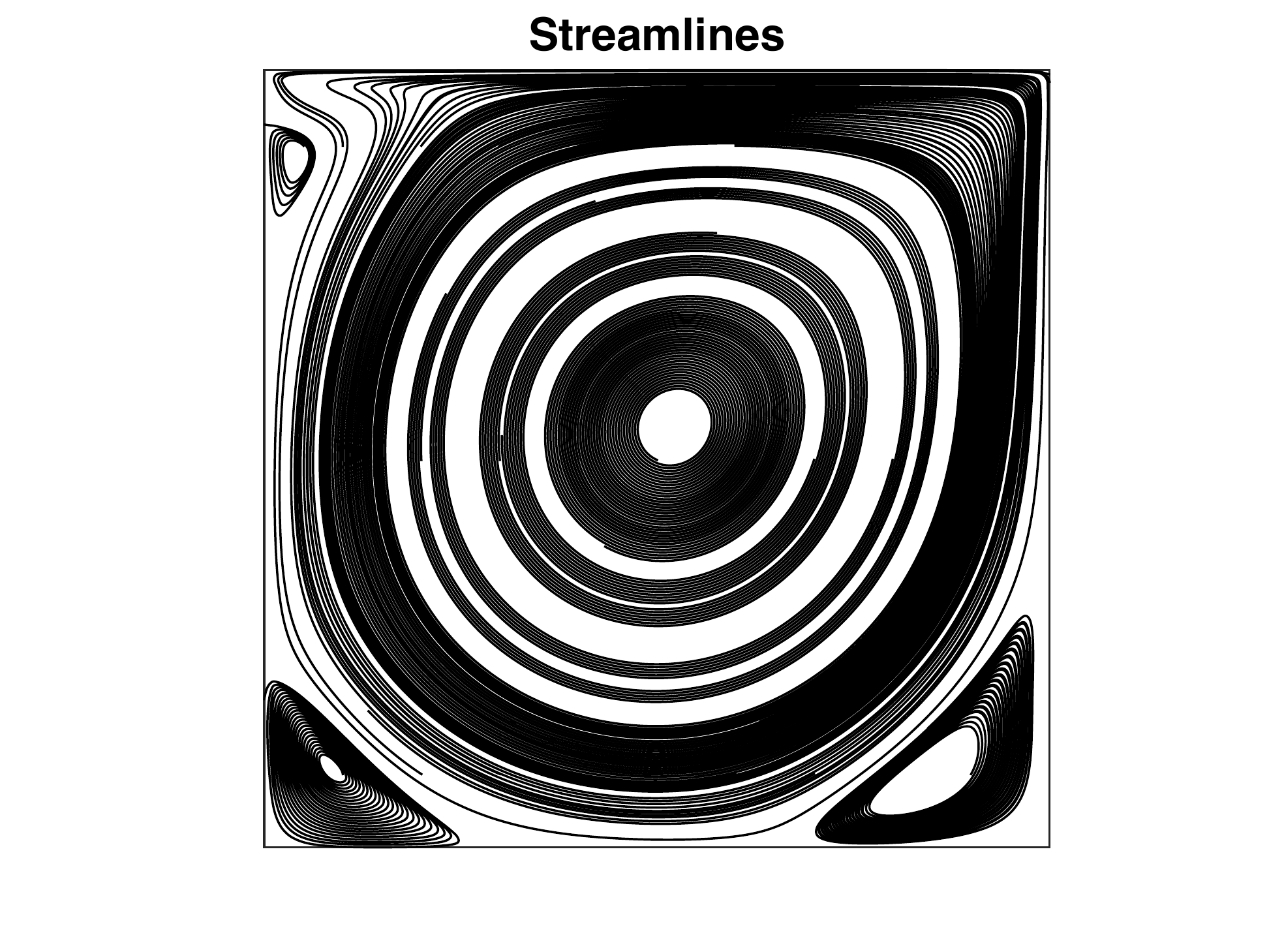}
\includegraphics[width = .32\textwidth, height=.3\textwidth,viewport=50 0 470 395, clip]{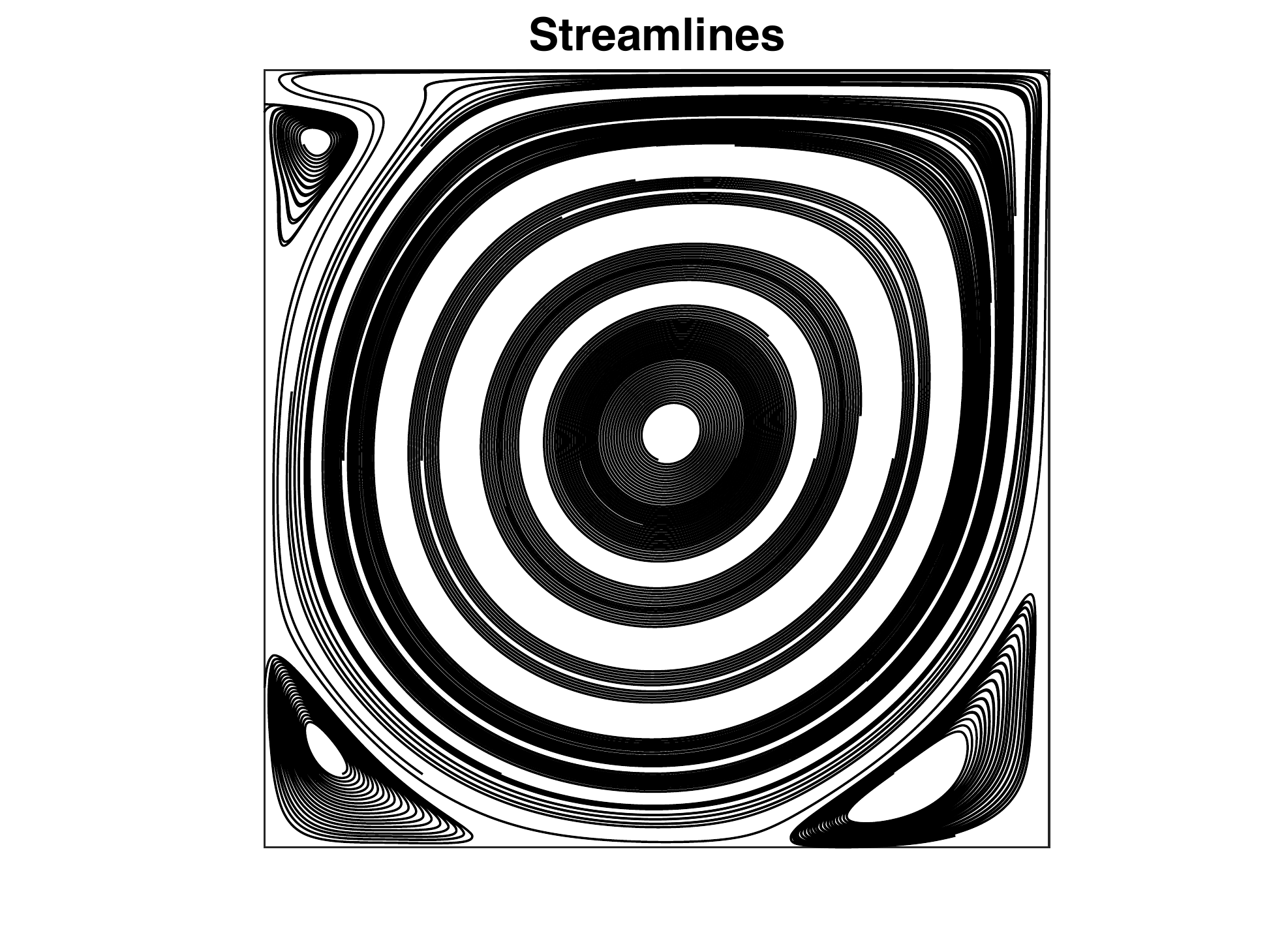}
	\caption{\label{2DCav2} Streamline plots of the solutions from 4 level Anderson accelerated Picard solvers at varying $Re$.}
	\end{center}
\end{figure}
The 2D driven cavity uses a domain $\Omega=(0,1)^2$, with no slip boundary conditions on the sides and bottom, and a `moving lid' on the top which is implemented by enforcing the Dirichlet boundary condition $u(x,1)=\langle 1,0 \rangle^T$.  There is no forcing ($f=0$), and the kinematic viscosity is set to be $\nu \coloneqq Re^{-1}$.  We discretize with $(P_2,P_1)$ Taylor-Hood elements on a $\frac{1}{64}$ mesh that provides 37,507 total degrees of freedom, and for the initial guess we used the Stokes solution on the same mesh and the same problem data. 
Plots of the velocity solutions from 4 level Anderson accelerated Picard solvers at 
$Re$=1000, 2500 and 5000 are shown in Figure \ref{2DCav2}, 
and these solutions match well those from recent literature \cite{bruneau:cavity}.

Convergence results for $Re=$1000, 2500, and 5000 are shown in Figure \ref{2DCav1}.  
In all cases, we observe an improvement from Anderson acceleration for the Picard method,
with an increase in improvement for higher Reynolds numbers.  
That is, while Anderson acceleration offers just a modest gain for $Re$=1000, 
for $Re$=2500 the gain is much greater, and for $Re$=5000, Picard appears to fail 
(or at least will take many, many iterations to converge to a reasonable tolerance).  
The Newton solver works very well for $Re$=1000, but fails for higher $Re$.  
We see the best Anderson performance in all cases with $m=4$, however 
the convergence behaviors with $m=3$ and $m=4$ are generally close, with $m=3$ more 
stable.
\begin{figure}[H]
\begin{center}
$Re$=1000 \hspace{1.25in} $Re$=2500 \hspace{1.45in} $Re=5000$ \\
\includegraphics[width = .32\textwidth, height=.3\textwidth,viewport=0 0 600 420, clip]{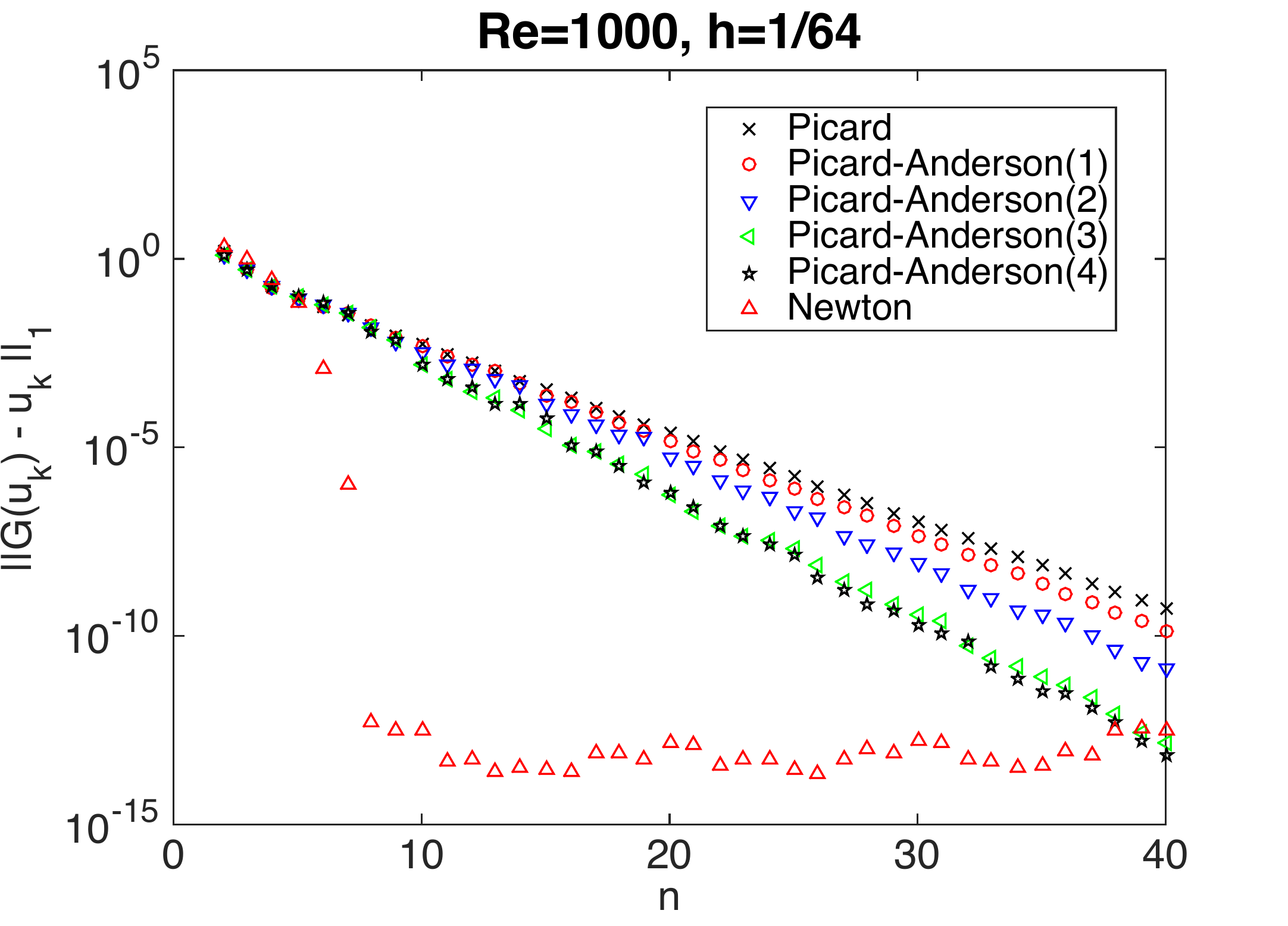}
\includegraphics[width = .32\textwidth, height=.3\textwidth,viewport=0 0 600 420, clip]{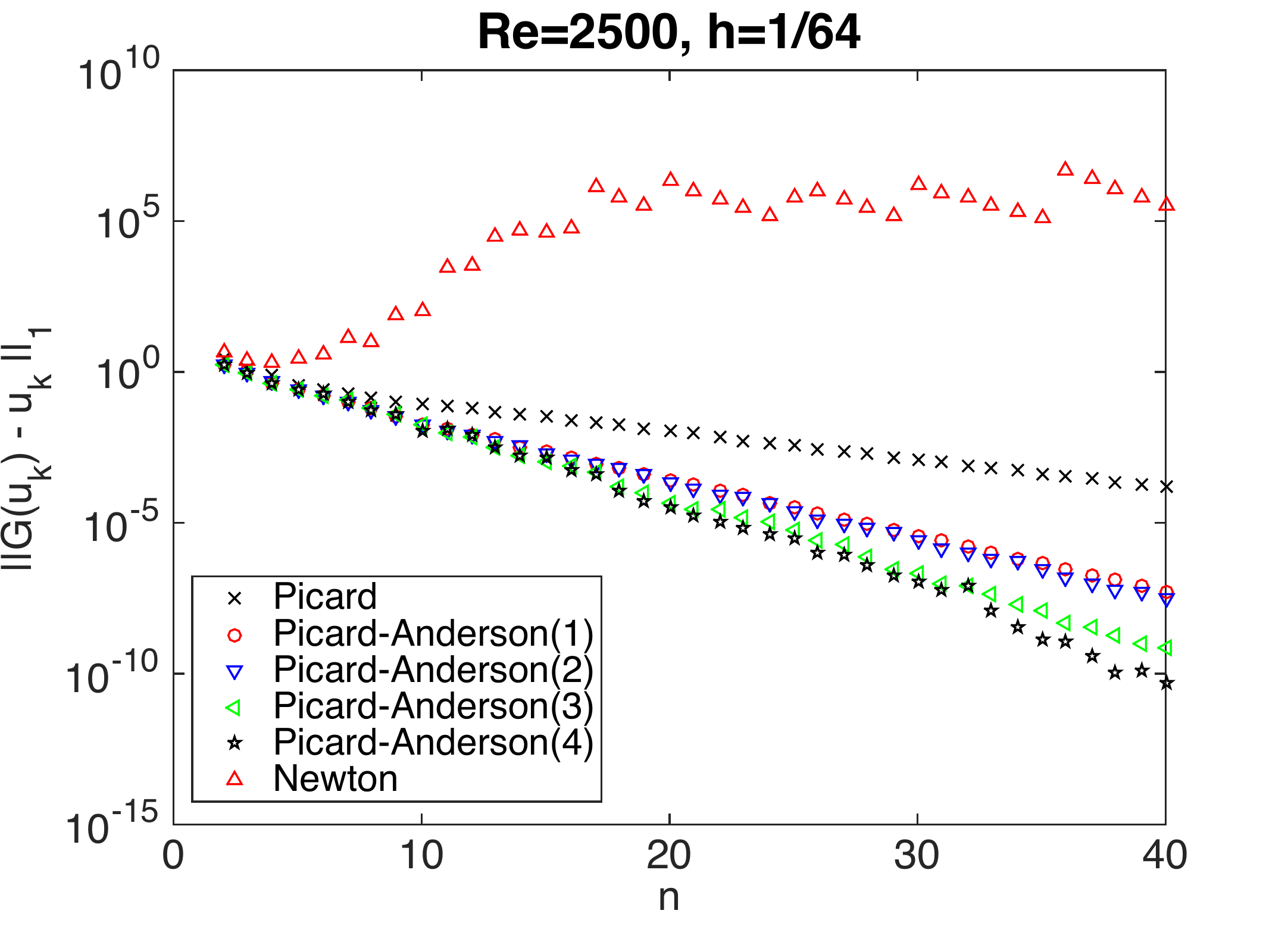}
\includegraphics[width = .32\textwidth, height=.3\textwidth,viewport=0 0 600 420, clip]{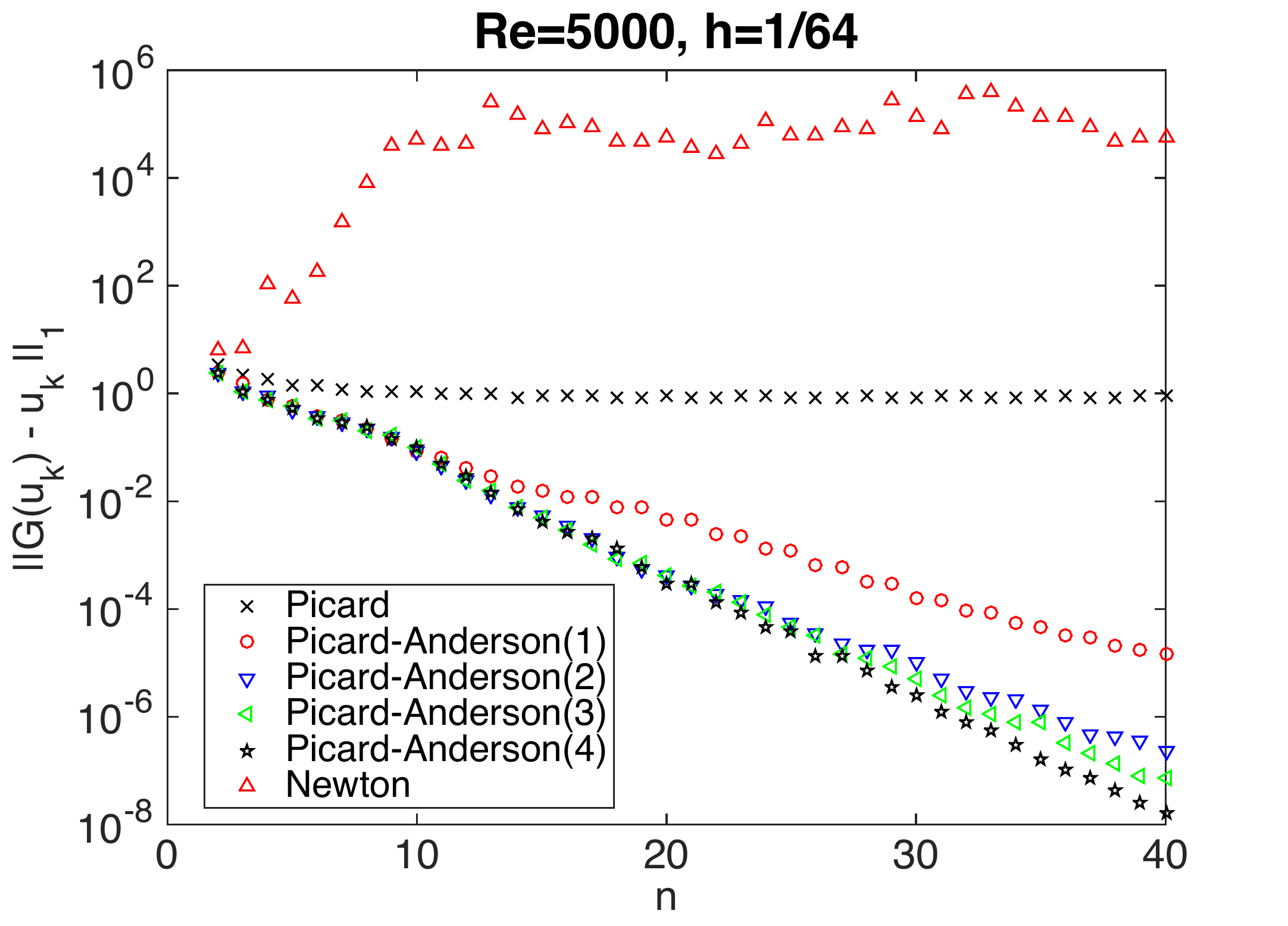}
	\caption{\label{2DCav1} Convergence of the various nonlinear solvers for the 2D Cavity test at varying $Re$.}
	\end{center}
\end{figure}
Figure \ref{fig:2DCav2-theta} shows the computed gain $\theta_k$ for
each optimization problem, for each value of $m$ and $Re$ investigated.  
Here we note the volatility in $\theta_k$ for the $m=4$ case, in agreement with the
instability in the convergence rate compared with $m=3$. In fact, we observe for each
Reynolds number at least one index $k$ for which $\theta_k > 1$ for $m=4$.
This suggests the source of the instability in the convergence rate is the 
failure of `fminsearch' to adequately solve the optimization problem for $m=4$.
Nonetheless, we generally see smaller values of $\theta_k$ (greater gain) with 
increasing $m$.  Notably, many of the $m=1$ values of $\theta_k$ are close to unity,
suggesting the importance of including search directions from earlier in the history.
\begin{figure}[H]
\begin{center}
$Re$=1000 \hspace{1.25in} $Re$=2500 \hspace{1.45in} $Re=5000$ \\
\includegraphics[width = .32\textwidth, height=.3\textwidth,viewport=0 0 600 420, clip]{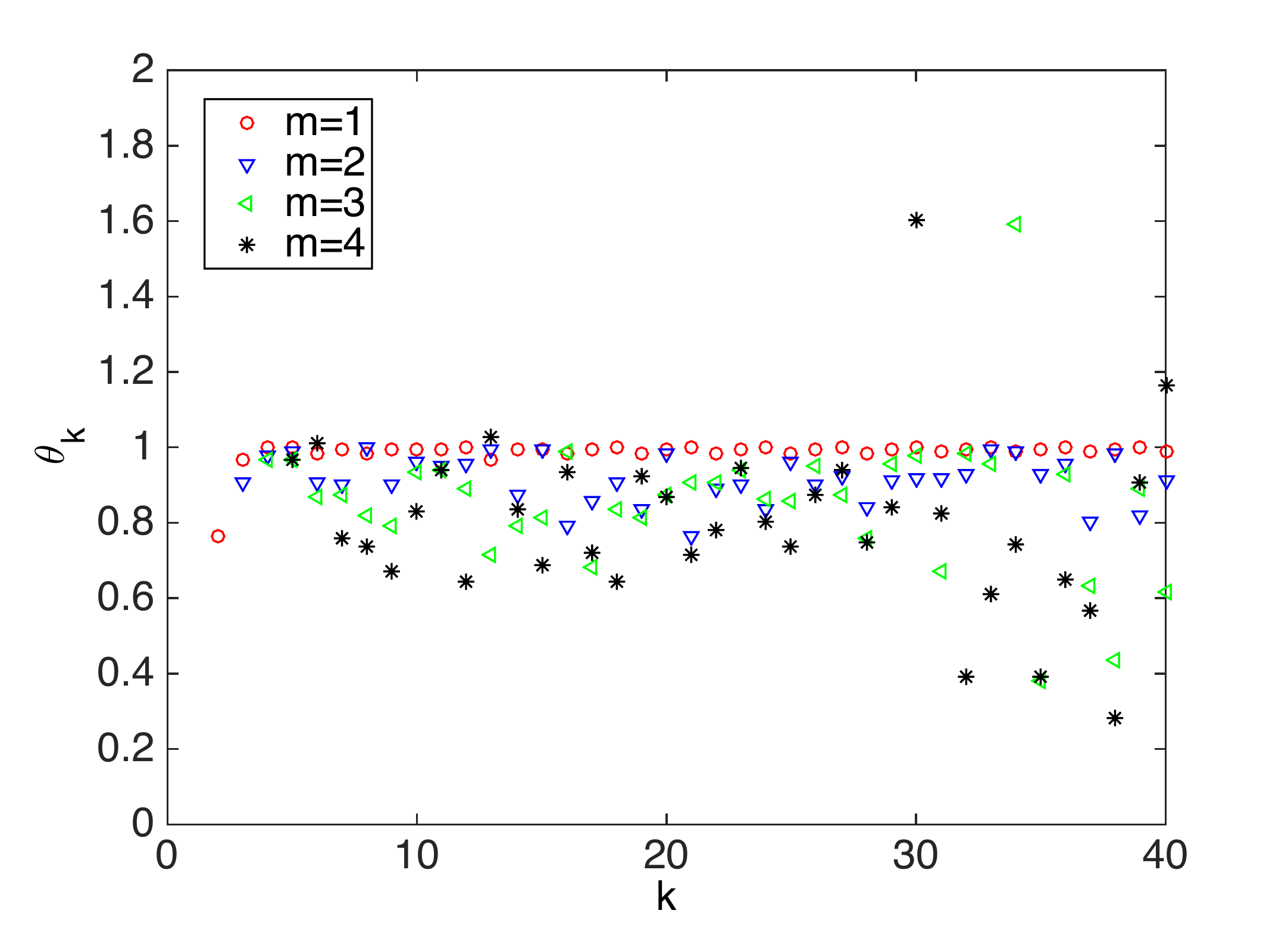}
\includegraphics[width = .32\textwidth, height=.3\textwidth,viewport=0 0 600 420, clip]{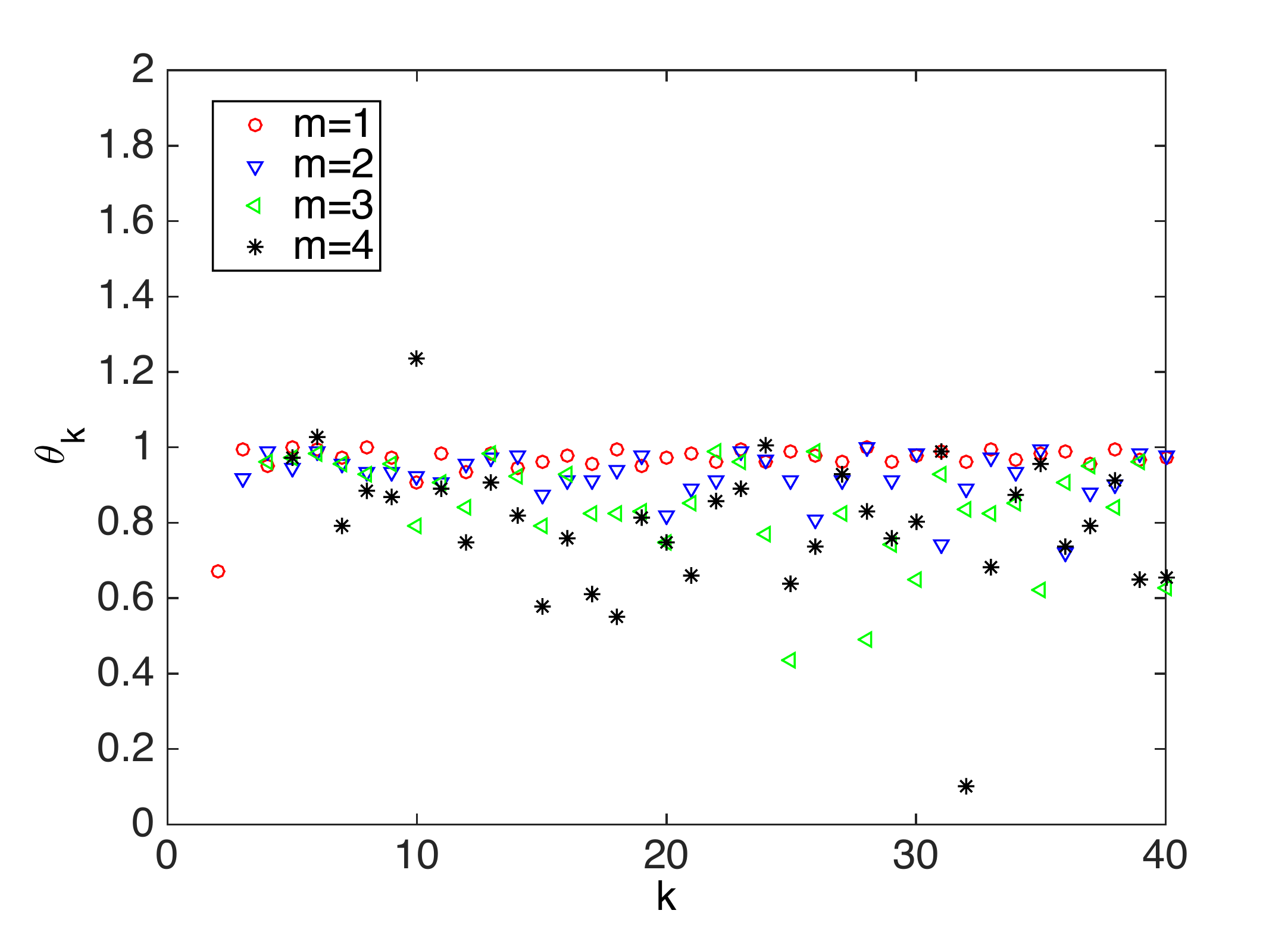}
\includegraphics[width = .32\textwidth, height=.3\textwidth,viewport=0 0 600 420, clip]{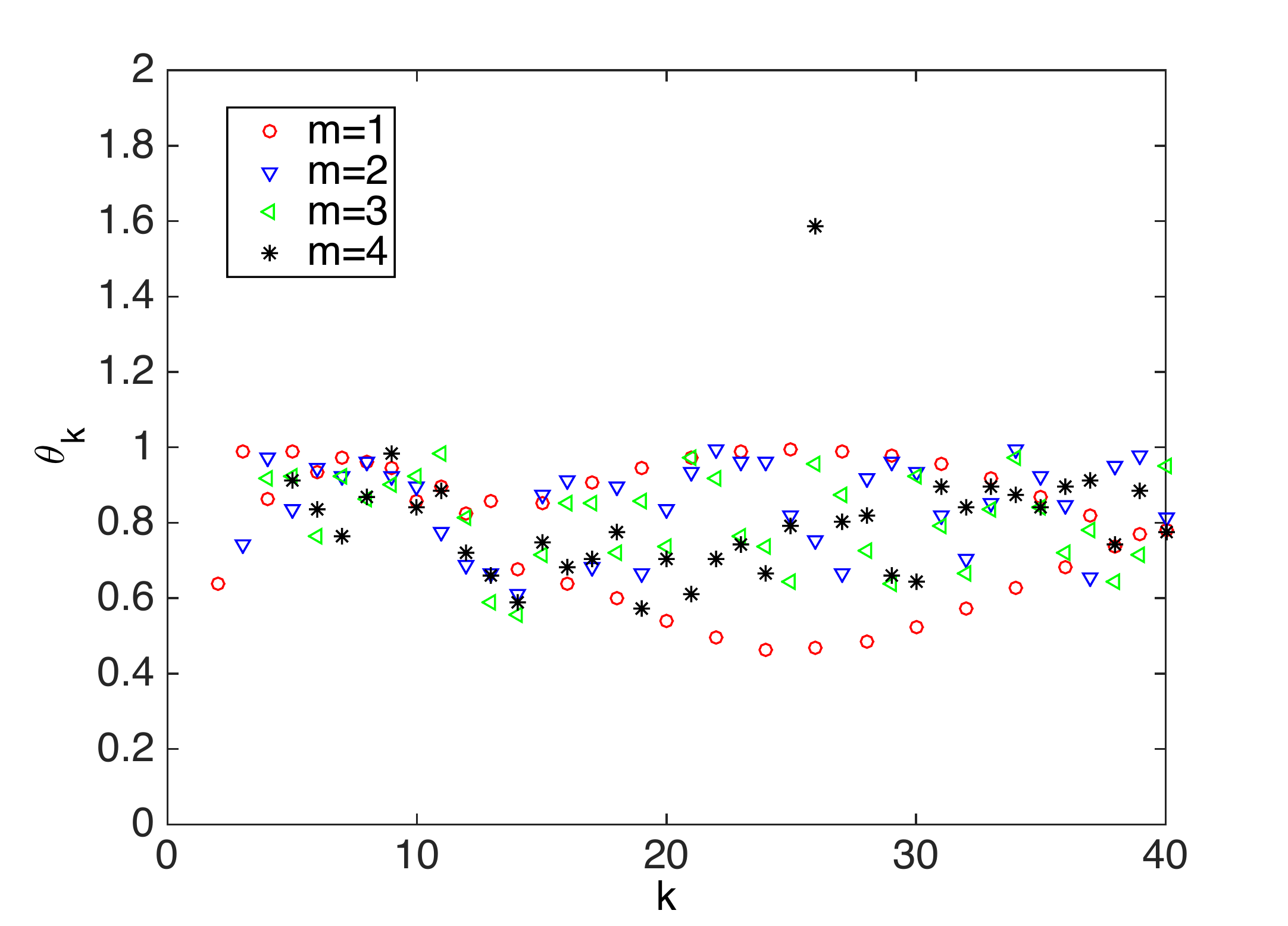}
\caption{\label{fig:2DCav2-theta} The $\theta_k$ vs. $k$, for varying $m$ for the 
  $Re$=2500 driven cavity simulation.}
\end{center}
\end{figure}
We compare our numerical results with the theoretical ones by comparing median values
of the gain $\theta_k$ and convergence rate $\kappa$  taken over all iterations $k$ 
for each $m$ and value of $Re$ investigated.  Table \ref{tpa01} shows the computed
$\theta_{med}$, and Table \ref{tpa02} compares the theoretical convergence rate
approximated to first order by $\kappa_{med}\theta_{med}$ to the computed mean 
convergence rate taken over all iterations.  We find the computed rates bounded 
below the theoretical ones, with a better prediction for lower values of $Re$.  
\begin{table}[H]
\begin{center}
 \begin{tabular}{c|ccc}
   & $Re$=1000 & $Re$=2500 & $Re$=5000 \\ \hline
 m & $\theta_{med}$ & $\theta_{med}$  & $\theta_{med}$   \\ \hline
 1 & 0.9936	& 0.9752  & 0.8503  \\
 2 & 0.9154 & 0.9282 &  0.8830 \\ 
 3 & 0.8719 & 0.8397 & 0.8164\\ 
 4 &  0.7902 & 0.7984 & 0.7738
\end{tabular}
\end{center}
\caption{\label{tpa01} Shown above are median values of $\theta_k$ for the 2D driven cavity simulations.
$\kappa_{med}$ is calculated from the Picard iteration above (without acceleration) to be 0.8040.  }
\end{table}

\begin{table}[H]
\begin{center}
 \begin{tabular}{c|cccccc}
   & $Re$=1000 & $Re$=1000 & $Re$=2500 & $Re$=2500 &$Re$=5000 & $Re$=5000  \\ \hline
 m & conv rate & $\theta_{med}^m \cdot 0.5848$ &  conv rate & $\theta_{med}^m \cdot 0.7951$  & conv rate &  $\theta_{med}^m \cdot 0.9696$   \\ \hline
0 & 0.5848 & -            &0.7951 & -       &0.9696 & - \\
 1 & 0.5471 & 0.5811 &    0.6423 & 0.7753 & 0.7270 & 0.8245 \\
 2 &0.5205  & 0.5353 &  0.6513  & 0.7380 & 0.6463  & 0.8562 \\ 
 3 & 0.4643 & 0.5099 &  0.5695 &  0.6676 &   0.6301 & 0.7916 \\ 
 4 &  0.4129 & 0.4621 &  0.5624 & 0.6348&  0.6121 &0.7503  \\
\end{tabular}
\end{center}
\caption{\label{tpa02} Shown above are median values of the convergence rates (median of successive difference ratios), and an estimate of the predicted rate of our theory, using the product of the median gain of the optimization $\theta_{med}^m$ with the median convergence rate of the Picard iteration, for varying $Re$ and $m$.}
\end{table}

\subsection{3D lid driven cavity}\label{subsec:3Dlid}

\begin{figure}[H]
\centering
\includegraphics[width = .89\textwidth, height=.28\textwidth,viewport=50 300 600 490, clip]{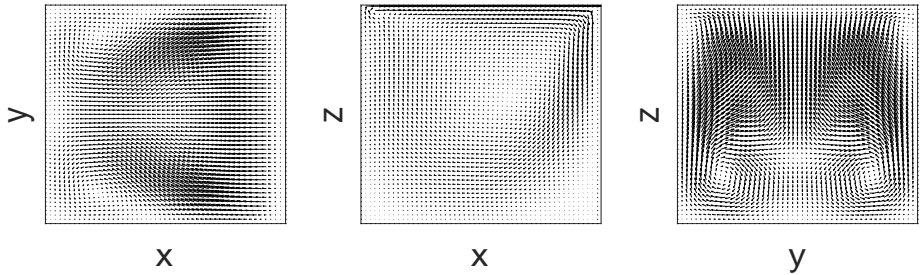}
\caption{\label{3dcavity2} Shown above are midsliceplane plots for the 3D driven cavity simulations at $Re$ = 400 using Picard-Anderson(4) method, these plots are well agreement with \cite{WB02}.}
\end{figure}

Next, we test AAPINSE on the 3D lid driven cavity problem. This problem is similar to the 2D case, and uses no slip boundary conditions on all walls, 
$u = \langle 1,0,0\rangle^T$ on the moving lid, no forcing, and set $\nu = \frac{1}{400}$.  We compute with $(P_3,P_2^{disc})$ Scott-Vogelius elements on a barycenter refined tetrahedral mesh that provides 796,722 total degrees of freedom.  We  tested the algorithm with different levels of optimization, all with initial guesses of zero in the interior but satisfying the boundary conditions.
Figure \ref{3dcavity2} shows a visualization of the computed solution
with $m=4$, which are in well agreement with \cite{WB02}.

\begin{figure}[H]
\centering
\includegraphics[width = .49\textwidth, height=.4\textwidth]{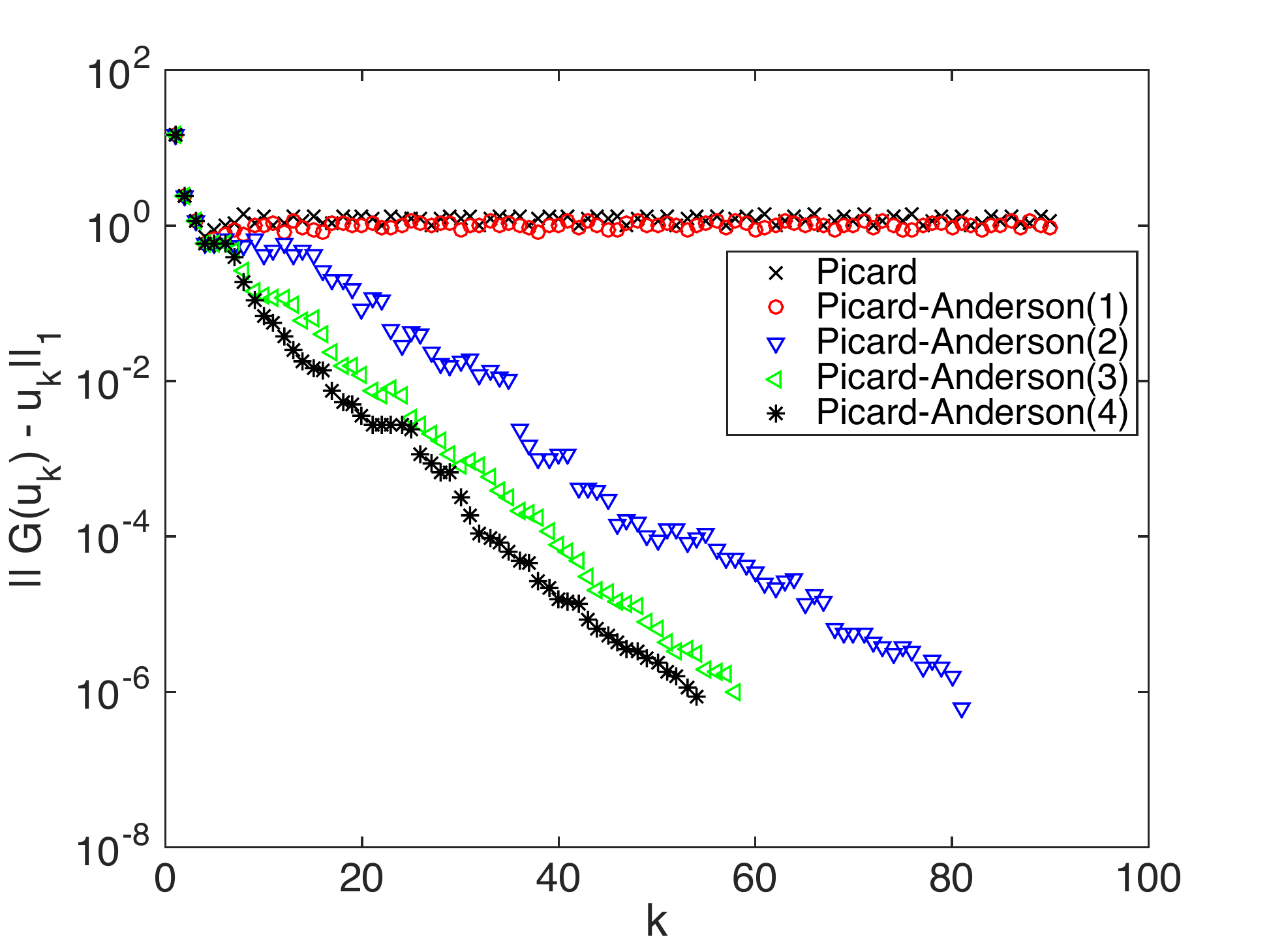}
\includegraphics[width = .49\textwidth, height=.4\textwidth]{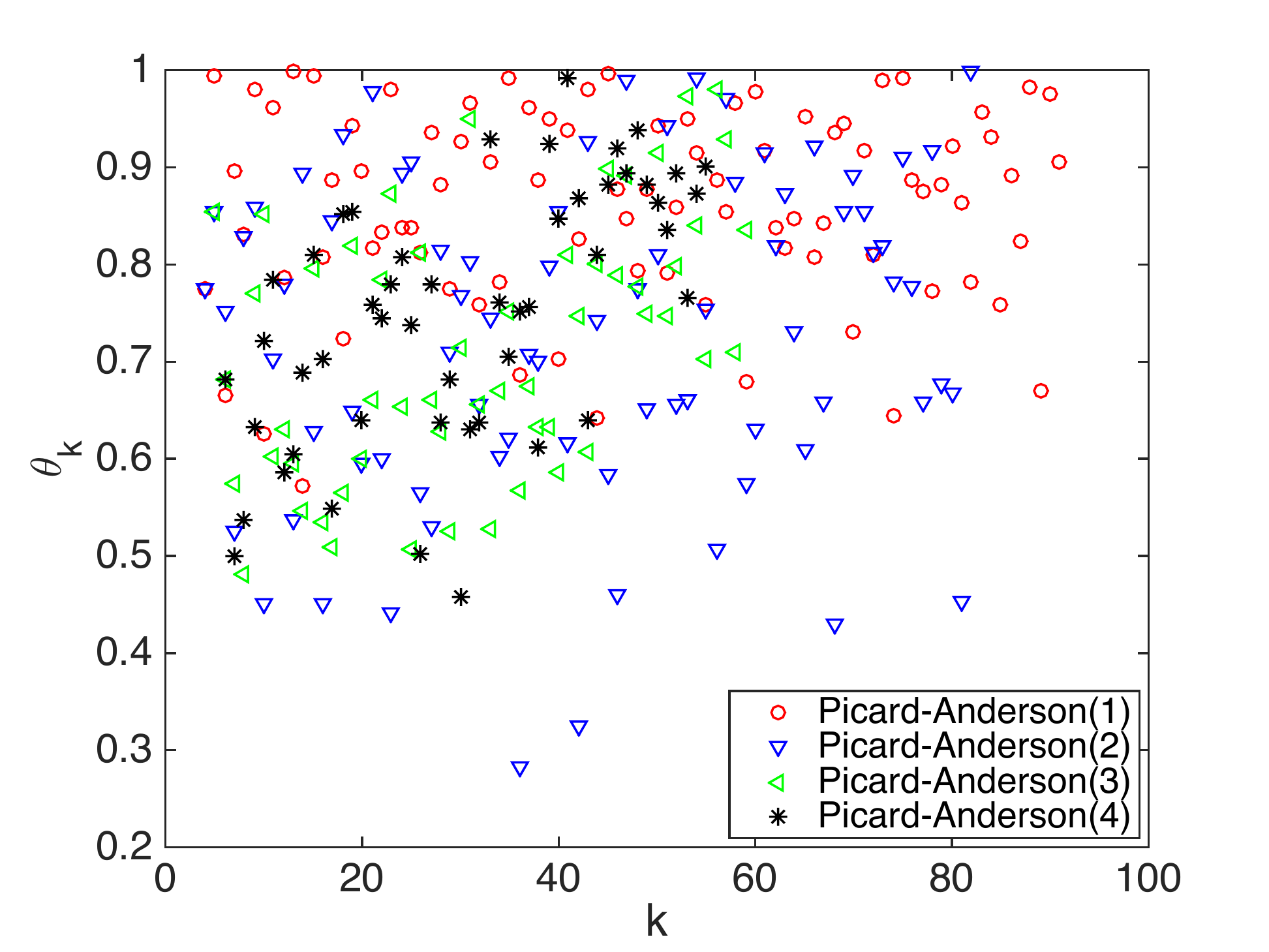}
\caption{\label{3dcavity} Convergence (left) and $\theta_k$ (right)  for AAPINSE for the 3D Cavity test at $Re= 400$.}
\end{figure}
Figure \ref{3dcavity} shows the convergence rate (left) and values of $\theta_k$ (right) for AAPINSE with varying $m$. From the convergence plot, we observe both Picard iteration and AAPINSE with $m=1$ fail, but a dramatic improvement is obtained using $m\ge 2$, sufficient to provide convergence. The $\theta_k$ plot shows the computed gain for each optimization problem and each value of $m$.  All $\theta_k$ values are below unity for this test, but for $m=1$ are closer to 1 than for larger $m$, however no gain is evident from the plot for increasing $m$ above 2.  This is also evident in Table \ref{tpa03}, which summarizes the computed median values of $\theta_k$ for the different $m$, and we observe lower values for $m\ge 2$, but no significant gain for choosing $m$ larger (in fact, $m=4$ gives slightly worse results than $m=3$, which we suspect is a result of `fminsearch' not exactly solving the optimization problem in this case). Table \ref{tpa04} compares the computed median convergence rate over all iterations to the theoretical convergence approximated by $\kappa_{med}\theta_{med}$.  
These results differ from the 2D case in that the computed rates are not bounded above
by the approximated theoretical rates (although for $m=4$ the values are close).  
This is expected given the convergence rate for the underlying fixed-point iteration
approximated by $\kappa_{med}=1.0215>1$ for this computation does not satisfy the 
small-data condition (the operator $G$ is not contractive).  
In particular, \eqref{eqn:lem2-e0} no longer implies the key estimates 
\eqref{eqn:ekalpha}-\eqref{eqn:ek1beta} in the $m=2$ case; 
and similarly \eqref{eqn:c2e3} does not imply \eqref{eqn:crit2} for the $m=1$ analysis.
\begin{table}[H]
\begin{center}
 \begin{tabular}{c|c}
   & $Re$=400  \\ \hline
 m & $\theta_{med}$    \\ \hline
 1 & 0.8612	 \\
 2 & 0.7281  \\ 
 3 & 0.7185 \\ 
 4 &  0.7508 
\end{tabular}
\end{center}
\caption{\label{tpa03}Shown above are median values of $\theta_k$ for the 3D driven cavity simulations. $\kappa_k$ is calculated from the Picard iteration above (without acceleration) to be 1.0215.}
\end{table}

\begin{table}[H]
\begin{center}
\begin{tabular}{c|cc}
m& conv rate &  $\theta_{med}^m \cdot 1.0215$ \\ \hline
0 & 1.0215 & - \\ 
1&0.9936 & 0.8797\\
2& 0.8623 & 0.7438 \\
3& 0.7967& 0.7340 \\
4& 0.7736 & 0.7670
\end{tabular}
\end{center}
\caption{\label{tpa04} Shown above are median values of the convergence rates (median of successive difference
ratios), and an estimate of the predicted rate of our theory, using the product of the median gain
of the optimization $\theta^m_{med}$ with the median convergence rate of the Picard iteration for varying $m$.}
\end{table}

\section{Conclusions}\label{sec:conc}
In this paper, we showed that Anderson acceleration applied to the Picard iteration can 
provide a significant, and sometimes dramatic, improvement in convergence behavior.  
We proved this analytically, and to our knowledge this is the first proof of Anderson 
acceleration providing (essentially) guaranteed improved convergence for a fixed point 
iteration, and in particular for a nonlinear fluid system.  
We also give results of several numerical tests that show the gains provided 
by Anderson acceleration for this problem can even be an enabling technology in the 
sense that it allows for convergence when both the Picard and Newton iterations fail.
The presented theory is based on characterizing the improvement in the
fixed-point convergence rate by the gain from the optimization problem.
While our numerics show the theoretical results somewhat underpredict the 
effectiveness of the acceleration strategy, they appear to capture the highest order 
effects.

Important future work includes extending these ideas to the recently proposed IPY variant of the Picard iteration for the steady NSE\cite{RVX17}, which has similar convergence properties of Picard but has linear systems that are much easier to solve. We also plan to explore whether Anderson acceleration be used to aid in the convergence of Newton iterations for steady NSE, since Newton tends to fail for higher $Re$.  Applying Anderson acceleration to steady multiphysics problems such as MHD may also be a fruitful pursuit.


\end{document}